\documentclass[11pt,a4paper,reqno]{amsart}
\pdfoutput=1

\usepackage{amsfonts}
\usepackage{amsmath,amssymb,amsthm,amsxtra,mathabx}
\usepackage{mathrsfs}
\usepackage{float}
\usepackage[colorlinks, linkcolor=blue,anchorcolor=Periwinkle, citecolor=Orange,urlcolor=Emerald]{hyperref}
\usepackage[usenames,dvipsnames,all]{xcolor}
\usepackage{enumitem}
\usepackage{geometry,array} 
\usepackage{graphicx}
\usepackage{subfigure}
\usepackage{bookmark}
\usepackage{tikz}
\usepackage{url}
\usepackage{dsfont}
\usepackage{pdflscape}
\usepackage[colorinlistoftodos]{todonotes}
\usepackage{tablists} \restorelistitem
\usepackage{cleveref}

\usetikzlibrary{decorations.markings}
\tikzset{->-/.style={decoration={  markings,  mark=at position #1 with
    {\arrow{>}}},postaction={decorate}}}
\tikzset{-<-/.style={decoration={  markings,  mark=at position #1 with
    {\arrow{<}}},postaction={decorate}}}
\usepackage{extarrows}
\usepackage{xcolor}   
\usepackage[all,pdf]{xy} 
\newcommand{\ind}{\operatorname{Ind}}
\newcommand{\oK}{\operatorname{K}}
\newcommand{\Hom}{\operatorname{Hom}}
\newcommand{\per}{\operatorname{per}}

\newcommand{\oZ}{\operatorname{Z}}
\newcommand{\HF}{\sqrt{\varrho}}
\newcommand{\HR}{\sqrt{\varrho}}
\newcommand{\CHR}{\sqrt{\varrho}}
\newcommand{\AHR}{\sqrt{\varrho}^{-1}}

\newcommand{\Cone}{\operatorname{Cone}}
\newcommand{\id}{\operatorname{id}}
\newcommand{\thick}{\operatorname{thick}}
\newcommand{\op}{\operatorname{op}}

\newcommand{\oX}{\operatorname{X}}
\newcommand{\oY}{\operatorname{Y}}
\newcommand{\OC}{\widetilde{\operatorname{OA}}(\mfS^{\lambda})}
\newcommand{\CC}{\widetilde{\operatorname{CA}}(\mfS^{\lambda})}
\newcommand{\tsigma}{\tilde{\sigma}}
\newcommand{\tgamma}{\tilde{\gamma}}
\newcommand{\tiota}{\tilde{\eta}}
\newcommand{\talpha}{\tilde{\gamma}}
\newcommand{\tbeta}{\tilde{\eta}}
\newcommand{\teta}{\tilde{\eta}}

\newcommand{\mfS}{\mathbf{S}}
\newcommand{\mfM}{\mathbf{M}}
\newcommand{\mfY}{\mathbf{Y}}
\newcommand{\mfU}{\mathbf{A}}
\newcommand{\mfV}{\mathbf{A^*}}
\newcommand{\mfp}{\mathbf{p}}
\newcommand{\mfA}{\mathbf{A}}
\newcommand{\mfR}{\mathbf{R}}
\newcommand{\mcH}{\mathcal{H}}
\newcommand{\dfd}{\operatorname{pvd}}
\newcommand{\pvd}{\operatorname{pvd}}
\newcommand{\mcC}{\mathcal{C}}
\newcommand{\mcD}{\mathcal{D}}
\newcommand{\mcK}{\mathrm{K}}
\newcommand{\mcHom}{\mathcal{H}om}
\newcommand{\mbZ}{\mathbb{Z}}

\newcommand{\mbP}{\mathbb{P}}
\newcommand{\mbD}{\mathbb{D}}
\def\bpt{node[blue]{$\bullet$}}
\def\rpt{node[white]{$\bullet$}node[red]{$\circ$}}
\newcommand{\sS}{\tilde{S}}
\newcommand{\rl}{\alpha}
\newcommand{\G}{\textcolor{Green}}
\newcommand{\Gid}{\textcolor{Green}{\id}}
\newcommand{\Po}{p}
\def\emb{\iota}
\def\pro{\pi}

\theoremstyle{plain}
\newtheorem{theorem}{Theorem}[section]

\newtheorem{lemma}[theorem]{Lemma}
\newtheorem{corollary}[theorem]{Corollary}
\newtheorem{proposition}[theorem]{Proposition}

\theoremstyle{definition}
\newtheorem{definition}[theorem]{Definition}

\newtheorem{example}[theorem]{Example}

\newtheorem{remark}[theorem]{Remark}
\newtheorem{notations}[theorem]{Notations}
\numberwithin{equation}{section}
\newtheorem{definition-proposition}[theorem]{Definition-Proposition}
\newtheorem{construction}[theorem]{Construction}

\newtheorem{convention}[theorem]{Convention}
\newcommand{\poly}[2]{\mbD^{#1}_{#2}}

\newcommand{\ver}[2]{V^{#1}_{#2}}
\newcommand{\wer}[2]{W^{#1}_{#2}}
\newcommand{\Rho}[2]{\rho^{#1}_{#2}}
\newcommand{\Zeta}[2]{\zeta^{#1}_{#2}}
\newcommand{\M}[2]{M^{#1}_{#2}}

\newcommand{\paths}[2]{p^{#1}_{#2}}
\newcommand{\K}[2]{k^{#1}_{#2}}

\def\k{\mathbb{K}}
\newcommand{\rest}[1]{\operatorname{Rest}(#1)}
\def\Int{\operatorname{Int}}

\begin{document}

\title[A geometric realization of Koszul duality]
{A geometric realization of Koszul duality for graded gentle algebras}

\date{\today}

\author{Zixu Li}
\address{Lizx: Department of Mathematical Sciences,
	Tsinghua University,
    100084 Beijing,
    China}
\email{lizx19@mails.tsinghua.edu.cn}

\author{Yu Qiu}
\address{Qy:
	Yau Mathematical Sciences Center and Department of Mathematical Sciences,
	Tsinghua University,
    100084 Beijing,
    China.
    \&
    Beijing Institute of Mathematical Sciences and Applications, Yanqi Lake, Beijing, China}
\email{yu.qiu@bath.edu}

\author{Yu Zhou}
\address{Zy:
	Yau Mathematical Sciences Center,
	Tsinghua University,
    100084 Beijing,
    China}
\email{yuzhoumath@gmail.com}

\begin{abstract}
We show that the Koszul functor of a homologically smooth graded gentle algebra can be realized as the half rotation in a geometric model. As a byproduct, we prove an intersection-dim formula involving the Koszul functor.
\end{abstract}

\keywords{Geometric model, Koszul duality, graded gentle algebras}
\maketitle


\section*{Introduction}

Geometric models play an important role in representation theory of algebras. 
The perfect derived categories of graded gentle algebras appeared in \cite{HKK} as the topological Fukaya categories of graded marked surfaces, 
where the indecomposable objects correspond to graded curves with local systems, and the space of stability conditions on the topological Fukaya categories is shown to be identified with the modulo space of meromorphic quadratic differentials. Later, it was shown that any gentle algebra is obtained in this way, and using this geometric model, an intersection-dim formula and a complete derived invariant for gentle algebras were given, cf. \cite{LP,OPS,APS19,O}. 
Recently, Qiu-Zhang-Zhou \cite{QZZ} generalized the geometric model to the graded skew-gentle case.
For more developments of geometric models of derived categories, we refer to \cite{AB,Am,CS,CS0,CJS,CHS,JSW,PPP18,Q16,Q18,QZ19,IQ2,IQZ}.

The Koszul duality is a duality phenomenon, which generalizes the duality between the symmetric algebra and the exterior algebra of a vector space. 
By studying the Koszul duality for Koszul rings, Beilinson-Ginzburg-Soergel introduced Koszul duality into representation theory in \cite{BGS96}. Later, Keller generalized the definition of Koszul duality for augmented differential graded categories in \cite{Ke94}. 

In this paper, we study a geometric realization of Koszul duality of homologically smooth graded gentle algebras. 
Let $\mfS^{\lambda}$ be a graded marked surface and $\mfA$ a full formal open arc system (see \Cref{def:f.f.a.s.}). There is an associated graded gentle algebra $\Lambda=\Lambda_{\mfA}$ (see \Cref{def:alg from as}). Let $\mfA^*$ be the dual closed arc system of $\mfA$. Then the associated graded gentle algebra $\Lambda_{\mfA^*}$  
is (quasi-isomorphic to) the Koszul duality of $\Lambda$ (see \Cref{prop:quad dual} and \Cref{lem:quasi koszul}). In the ungraded case, \Cref{prop:quad dual} is known in \cite{OPS}, see \cite{L-FSV} for a similar result about ungraded skew-gentle algebras.

Following \cite{QZZ}, there is an injective map (see \Cref{cons:obj})
$$\oX:\OC \rightarrow \per\Lambda_{\mfA},$$
where $\OC$ is the set of graded open arcs on $\mfS^{\lambda}$ (see \Cref{def:arc}), and $\per\Lambda_{\mfA}$ is the perfect derived category of $\Lambda_{\mfA}$. Dually, there is an injective map
$$\oY:\CC \rightarrow \per\Lambda_{\mfA^*},$$
where $\CC$ is the set of graded closed arcs on $\mfS^{\lambda}$. We introduce an operator $\HR$ on graded arcs, called half rotation (see \Cref{def:half rot}), whose square is the usual rotation.

The main result of the paper is the following geometric realization of the Koszul functor $\mcK_{\mfA}:\per\Lambda_{\mfA}\to\per\Lambda_{\mfA^*}$.

\begin{theorem}[\Cref{thm:int=dim}]
For any graded open arc $\tiota$, there is an isomorphism
$$\mcK^{{\mfA}}(\oY(\tiota)) \cong \oX(\CHR(\tiota)),$$
in $\per\Lambda_{\mfA}$, where $\mcK^{\mfA}$ is a quasi-inverse of $\mcK_{\mfA}$. That is, there is a commutative diagram
\[
\xymatrix@C=3pc{
\CC  \ar[d]_{\CHR} \ar[r]^{\oY}& \per(\Lambda_{\mfA^*})\ar[d]^{\mcK^{{\mfA}}}   \\
\OC \ar[r]_{\oX}&   \per \Lambda_{\mfA}.
}
\]
\end{theorem}

We also prove the following intersection-dimension formula.
\begin{theorem}[\Cref{thm:koszul}]
Let $\tsigma$ be a graded open arc and $\tiota$ be a graded closed arc. Then for any $\rho \in \mbZ$, we have
\begin{equation} \notag
\Int^{\rho}(\tsigma,\tiota)=\operatorname{dim}\Hom_{\per\Lambda_{\mfA}}(\oX(\tsigma),\oK^{\mfA}(\oY(\tiota)){[\rho]}),
\end{equation}
where $\Int^{\rho}(\tsigma,\tiota)$ is the number of oriented intersections from $\tsigma$ to $\tiota$ of index $\rho$ (see \Cref{def:int}).
\end{theorem}

The paper is organized as follows. In \Cref{sec:surf and alg}, we recall basic notions and notations on graded marked surfaces and graded gentle algebras. In \Cref{sec:geo der eq}, we review a geometric model for perfect derived categories of graded gentle algebras. We show that this model is compatible with the smoothing of a sequence of arcs along oriented intersections (\Cref{thm:smoothing}), and is compatible with a derived equivalence between graded gentle algebras from different full formal open arc systems (\Cref{thm:independence}). In \Cref{geometric model for koszul duality}, we introduce the clockwise half rotation, which is used to give a geometric realization of the Koszul functor (\Cref{thm:koszul}). As an application, an intersection-dimension formula is also shown (\Cref{thm:int=dim}). 
 \subsection*{Conventions}
In this paper, $\k$ denotes a field. A quiver $Q$ consists of a vertex set $Q_0$, an arrow set $Q_1$, and two maps $s,t:Q_1\to Q_0$ giving the start and terminal of an arrow, respectively. For $a,b\in Q_1$, we denote by $ab$ the path in $Q$ first $a$ then $b$. For any two morphisms $f$ and $g$, the composition $gf=g\circ f$ means for first $f$ then $g$. 
\subsection*{Acknowledgments}
This work is supported by National Key R\&D Program of China (No. 2020YFA0713000) and
National Natural Science Foundation of China (Grant No. 12031007 and     No. 12271279).

\section{Graded marked surfaces and graded gentle algebras}\label{sec:surf and alg}


We follow \cite{HKK}. A \emph{graded marked surface} $\mfS^{\lambda}$ is a compact oriented surface $\mfS$ with
\begin{itemize}
    \item a non-empty boundary $\partial \mfS$,
    \item two finite sets $\mfM$ and $\mfY$ of points on $\partial \mfS$ such that each connected component of $\partial \mfS$ contains the same number (at least one) of marked points from $\mfM$ and $\mfY$, and these marked points are arranged alternately, and
    \item a section $\lambda:\mfS \rightarrow \mbP T\mfS$ of the real projectivization $\mbP T\mfS$ of the tangent bundle of $\mfS$.
\end{itemize}
Points in $\mfM$ (resp. $\mfY$) are called \emph{open} (resp. \emph{closed}) \emph{marked} points. From now on, we fix a graded marked surface $\mfS^\lambda$.




A \emph{graded arc} on a graded surface $\mfS^{\lambda}$ is defined as a pair $(\sigma,\tsigma)$, where 
\begin{itemize}
    \item $\sigma:[0,1]\rightarrow\mfS$ is an immersion, and
    \item $\tsigma$ is a family of (homotopy classes of) paths in $\mathbb{P}T_{\sigma(t)}\mfS$ from $\lambda(\sigma(t))$ to $\dot{\sigma}(t)$, varying continuously with $t\in(0,1)$.
\end{itemize}
We simply denote a graded arc $(\sigma,\tsigma)$ by $\tsigma$. The shift $\tsigma[1]$ is a graded arc whose underlying arc is still $\sigma$ and whose grading is the composition of the grading of $\tsigma$ with the clockwise rotation of $\pi$.

\begin{definition}\label{def:oint}
Let $\tsigma_1$ and $\tsigma_2$ be two graded arcs in a minimal position on $\mfS^{\lambda}$. An \emph{oriented intersection} from $\tsigma_1$ to $\tsigma_2$ is a pair $(t_1,t_2)\in[0,1]^2$ such that $\sigma_1(t_1)=\sigma_2(t_2)$ and there is a small arc $\alpha \subseteq \mfS^{\circ}=\mfS \setminus \partial \mfS$ around $\Po=\sigma_1(t_1)=\sigma_2(t_2)$ from a point in $\tsigma_1$ to a point in $\tsigma_2$ clockwise. Sometimes, we use the small arc $\alpha$ to denote the corresponding oriented intersection. Note that we need to identify the vertical angles if the intersection ${\Po} \in \mfS^{\circ}$. See \Cref{fig:oint}.
\begin{figure}[htpb]
\begin{tikzpicture}[scale=.5]
	\draw[thick, Green,->-=.5,>=stealth](-1,1)to[bend left=30](1,1);
	\draw[thick, Green,-<-=.5,>=stealth](-1,-1)to[bend left=-30](1,-1);
\draw[Green,thick]
(0,-1.9)node{$\alpha$}
(0,01.9)node{$\alpha$};
\draw[thick]
(1,0)node{${\Po}$}
;
	\draw[cyan, thick](2,2)to(-2,-2)(-2,2)to(2,-2);
	\draw[cyan,thick](-3,2)node{$\tsigma_1$}
(2.9,2)node{$\tsigma_2$};
    \begin{scope}[shift={(12,-2)}]
    \fill[fill=gray!10]
(-3,0) rectangle (3,-0.6);
	\draw[thick, Green,->-=.5,>=stealth](-1.7,1.7)to[bend left=30](1.7,1.7);
	\draw[Green, thick]
(-0.1,2.8)node{$\alpha$};
\draw[thick]
(0,-.5)node{${\Po}$}
;
    \draw[thick](3,0)to(-3,0);
	\draw[cyan,thick](4,4)to(0,0) to(-4,4);
	\draw[cyan,thick](-3,2.5)node{$\tsigma_1$}
(3.2,2.5)node{$\tsigma_2$};
    \end{scope}	
    \end{tikzpicture}
\caption{Oriented intersections between graded arcs.}
\label{fig:oint}
\end{figure}
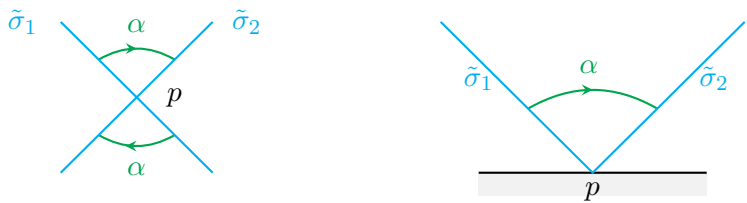

\end{definition}




\begin{definition}\label{def:int}
Let $\rl$ be an oriented intersection from $\tilde{\sigma}_1$ to $\tilde{\sigma}_2$ at $\Po=\sigma_1(t_1)=\sigma_2(t_2)$. 
If $\Po\in\mfS^\circ$, the \emph{intersection index} $\ind_{\Po}(\tsigma_1,\tsigma_2)$ is defined to be
$$\ind_{\Po}(\tsigma_1,\tsigma_2)=\tsigma_{1}(t_1) \cdot \kappa \cdot {\tsigma_{2}}^{-1}(t_2)\in \pi_{1}(\mbP(T_{\Po}\mfS))\cong \mbZ,$$
where $\kappa$ is obtained by a clockwise rotation in $\mathbb{P}(T_{\Po}\mathbf{S})$ by an angle less than $\pi$. 

If $\Po\in\partial\mfS$, fix an arbitrary grading $\tilde{\alpha}$ on the small arc $\alpha$. We define
\[
\mathrm{ind}_{\Po}(\tilde{\sigma}_1,\tilde{\sigma}_2)=\mathrm{ind}_{\alpha(0)}(\tilde{\sigma}_1,\tilde{\alpha})-\mathrm{ind}_{\alpha(1)}(\tilde{\sigma}_2,\tilde{\alpha}).
\]
Denote by $\overrightarrow{\cap}^{\rho}(\tsigma_1,\tsigma_2)$ (resp. $\Int^\rho(\tsigma_1,\tsigma_2)$) the set (resp. number) of oriented intersections from $\tsigma_1$ to $\tsigma_2$ with index $\rho$.
\end{definition}

In the case $p\in\mfS^\circ$, we have the following well-known formulas, cf. \cite[Sec.~2.1]{HKK}.
\begin{equation}\label{eq:hkk1}
    \ind_{\Po}(\tsigma_1,\tsigma_2)+\ind_{\Po}(\tsigma_2,\tsigma_1)=1,
\end{equation}
$$\ind_{\Po}(\tsigma_{1}[m],\tsigma_{2}[n])=\ind_{\Po}(\tsigma_1,\tsigma_2)+m-n.$$



\begin{definition}\label{def:arc}
A \emph{graded open (resp. closed) arc} $\tilde{\sigma}$ is a graded arc whose endpoints are in $\mfM$ (resp. $\mfY$). The set of (homotopy classes of) graded open (resp. closed) arcs on $\mfS^{\lambda}$ is denoted by $\OC$ (resp. $\CC$).
\end{definition}

Two graded arcs are said \emph{non-crossing} if there is no intersection between them in $\mfS^{\circ}$. A graded arc is called \emph{simple} if it has no self-intersections in $\mfS^\circ$.

\begin{definition}\label{def:f.f.a.s.}
    A \emph{full formal open (resp. closed) arc system} $\mfU$ is a collection of pairwise non-crossing graded simple open arcs, such that $\mfU$ divides $\mfS$ into polygons, called $\mfU$-polygons, each of which contains exactly one closed (resp. open) marked point.
\end{definition}

\begin{example}\label{exm:as}
In \Cref{fig:as}, there is a graded marked surface with three boundary components and genus zero. The blue points are open marked points, and the red points are closed-marked points. There is a full formal open arc system $\mfU$ consisting of 7 blue arcs, and a full formal closed arc system $\mfV$ consisting of 7 red arcs.
\begin{figure}[htpb]\centering
\begin{tikzpicture}[scale=.9]
\draw[ultra thick](0,0) circle (3.5);	
\filldraw
[fill=gray!10,ultra thick](1.5,0) circle (0.8);
\filldraw
[fill=gray!10,ultra thick](-1.5,0) circle (0.8);
\draw[blue,thick]
(-0.7,0)to(0.7,0)
(-0.7,0)to(0,3.5)
(-0.7,0)to(0,-3.5);
\draw[blue,thick]
(-2.3,0)to[out=95,in=-140](0,3.5)
(-2.3,0)to[out=-95,in=140](0,-3.5);
\draw[blue,thick]
(2.3,0)to[out=85,in=-40](0,3.5)
(2.3,0)to[out=-85,in=40](0,-3.5);
\draw
(-1.4,1.8)node[blue]{$\tgamma_1$}
(.1,2.5)node[blue]{$\tgamma_2$}
(1.45,1.8)node[blue]{$\tgamma_3$}
(1.45,-1.8)node[blue]{$\tgamma_4$}
(-0.2,0.25)node[blue]{$\tgamma_5$}
(.1,-2.5)node[blue]{$\tgamma_6$}
(-1.4,-1.8)node[blue]{$\tgamma_7$};
\draw[red,thick]
(-3.42,0.05)to[out=35,in=-185](-1.6,0.8)
(-3.42,0.05)to[out=-35,in=-175](-1.6,-0.8)
(-1.44,0.85)to[out=30,in=140](1.5,0.9)
(-1.44,-0.85)to[out=-30,in=-140](1.5,-0.87);
\draw[red,thick]
(1.43,0.85)to[out=-180,in=90]
(0.3,0)to[out=-90,in=-180]
(1.43,-0.85);
\draw[red,thick]
(1.57,0.84)to[out=0,in=140](3.45,0.1)
(1.57,-0.84)to[out=0,in=-140](3.47,-0.08);
\draw
(-2.8,0.8)node[red]{$\teta_1$}
(0.2,1.65)node[red]{$\teta_2$}
(2.8,0.85)node[red]{$\teta_3$}
(2.8,-0.85)node[red]{$\teta_4$}
(0.3,-0.6)node[red]{$\teta_5$}
(0.2,-1.65)node[red]{$\teta_6$}
(-2.8,-0.8)node[red]{$\teta_7$};
\draw
(0,3.48)\bpt (0,-3.52)\bpt
(2.3,0)\bpt (0.7,0)\bpt
(-2.3,0)\bpt (-0.7,0)\bpt
(-3.5,0)\rpt (3.5,0)\rpt
(1.5,0.8)\rpt (-1.5,0.8)\rpt
(1.5,-0.8)\rpt (-1.5,-0.8)\rpt;
\end{tikzpicture}
 \caption{Full formal open (closed) arc system.} 	
 \label{fig:as}
	\end{figure}
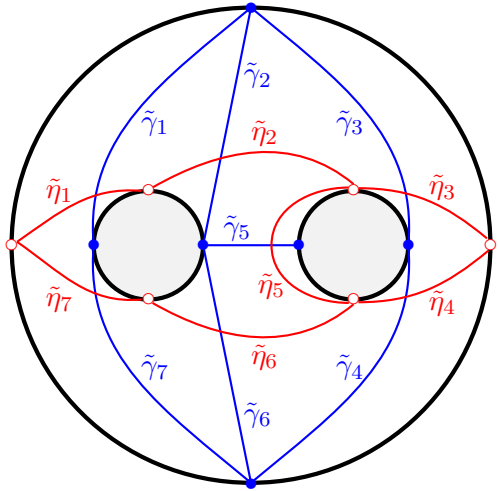
 \end{example}

A graded quiver is a quiver $Q=(Q_0,Q_1,s,t)$, endowed with a map $|\cdot|:Q_1\to\mathbb{Z}$. A relation set $I$ on $Q$ is a set of linear combinations of paths on $Q$. For any graded quiver $Q$ with relation $I$, the algebra $\k Q/\langle I\rangle$ is a graded algebra, where $\k Q$ is the path algebra and $\langle I\rangle$ is the ideal of $\k Q$ generated by $I$.

\begin{definition}\label{def:alg from as}
Let $\mfA= \{\talpha_1, \ldots, \talpha_n\}$ be a full formal open/closed arc system on $\mfS^{\lambda}$. The graded quiver $Q_{\mfU}$ with relation $I_{\mfU}$ associated to $\mfA$ is defined as follows.
\begin{itemize}
	\item The vertex set $(Q_{\mfU})_0=\{1,2,\cdots,n\}$ of $Q_{\mfU}$ indexed by open arcs $\talpha_1,\cdots, \talpha_n$.
	\item There is an arrow $a: i\rightarrow j$ whenever there is an oriented intersection $\alpha$ from $\talpha_j$ and $\talpha_i$ such that $\alpha$ is an inner angle of an $\mfU$-polygon. The degree $|a|=\ind_{\rl}(\talpha_j, \talpha_i)$.
	\item The relation set $I_{\mfU}$ consists of $ab$ whenever $a$ and $b$ are induced by the oriented intersections $\alpha$ and $\beta$ respectively, such that $\alpha$ and $\beta$ are adjacent inner angles of the same $\mfU$-polygon.
\end{itemize}
The graded algebra $\k Q_{\mfU}/\langle I_{\mfU}\rangle$ is denoted by $\Lambda_{\mfU}$.
\end{definition}



\begin{example}\label{exm:gen alg}
In \Cref{fig:as}, we obtain a graded quiver $Q_{\mfU}$ from $\mfU$, and a graded quiver $Q_{\mfV}$ from $\mfV$.
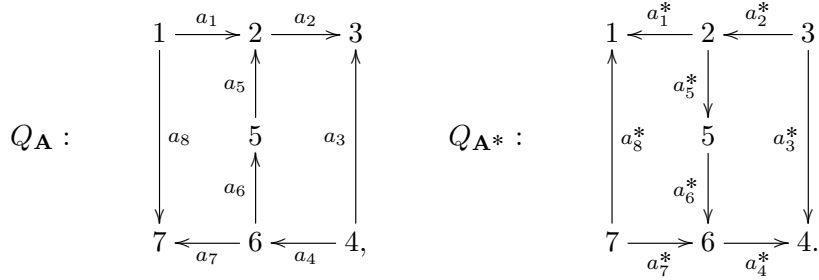
\begin{gather}\notag
	\xymatrix{
	&1 \ar[r]^{a_1} \ar[dd]^{a_8} & 2 \ar[r]^{a_2} & 3 && 1 \ar@{<-}[r]^{a_1^*} \ar@{<-}[dd]^{a_8^*} & 2 \ar@{<-}[r]^{a_2^*} & 3 \\
	Q_{\mfA}:&& 5 \ar[u]^{a_5} & &Q_{\mfV}:& &5 \ar@{<-}[u]^{a_5^*} & \\
&	7  & 6 \ar[l]^{a_7} \ar[u]^{a_6}  & 4, \ar[l]^{a_4} \ar[uu]^{a_3} && 7  & 6 \ar@{<-}[l]^{a_7^*} \ar@{<-}[u]^{a_6^*}  & 4.\ar@{<-}[l]^{a_4^*} \ar@{<-}[uu]^{a_3^*}
	}
\end{gather}
The relation sets $I_{\mfU}=\{ a_4a_6,a_5a_2\}$ and $I_{\mfV}=\{ a_2^*a_1^*,a_5^*a_6^*,a_7^*a_4^*\}$. Hence we obtain two graded algebras $\k Q_{\mfU}/\langle I_{\mfU}\rangle$ and $\k Q_{\mfV}/\langle I_{\mfV}\rangle$.
\end{example}

\begin{remark}\label{rmk:rebuilt}
    Denote by $p_{\rl}$ the arrow in $Q_{\mfU}$ corresponding to the oriented intersection $\rl$, which is an inner angle of an $\mfU$-polygon. By the construction in \Cref{def:alg from as}, for any such oriented intersections $\rl_1,\cdots,\rl_s$, they are composable if and only if the path $p_{\rl_s}\cdots p_{\rl_1}$ is not in the ideal $\langle I_{\mfU}\rangle$. Thus, one can identify the graded algebra $\Lambda_{\mfU}$ as a $\k$-space with the oriented intersections $\rl$ between arcs in $\mfU$ as a basis, the index of $\rl$ as its degree, and the composition of oriented intersections as the multiplication.

\end{remark}

For any full formal open arc system $\mfU= \{\talpha_1, \ldots, \talpha_n\}$, there is a unique dual full formal closed arc system $\mfU^\ast=\{\tbeta_{1}, \cdots,\tbeta_{n}\}$, in the sense that
\[\Int^{\rho}(\talpha_i,\tbeta_j)=\begin{cases}
    1&\text{if $i=j$ and $\rho=0$,}\\
    0&\text{otherwise.}
\end{cases}\]
See e.g., \cite[Prop. 1.16]{OPS}.
We say that the simple arcs $\talpha_i$ and $\tbeta_i$ are dual to each other (w.r.t. $\mfU$). In \Cref{exm:as}, if we require that $\ind(\talpha_i, \tbeta_i)=0$, then the two full formal closed arc systems there are dual to each other.


From now on, we fix a full formal open arc system $\mfA= \{\talpha_1, \ldots, \talpha_n\}$ and its dual arc system $\mfV=\{\tbeta_{1}, \cdots,\tbeta_{n}\}$. 

\begin{convention}
In the figures,
\begin{enumerate}
	\item open marked points in $\mfM$ and arcs in $\mfA$ are drawn in blue,
	\item other open arcs are drawn in cyan,
    \item closed marked points in $\mfY$ and arcs in $\mfV$ are drawn in red,
    \item other closed arcs are drawn in magenta,
    \item oriented intersections between two graded arcs are drawn in green.
\end{enumerate}
\end{convention}

The following lemma is a direct observation.
\begin{lemma}\label{lem:deg}
 Let $M \in \mfM$ be an intersection of $\talpha_{i}$ and $\talpha_{j} \in \mfA$ and let $Y \in \mfY$ be an intersection of $\tbeta_{i}$ and $\tbeta_{j} \in \mfV$, such that the relative position of $\tgamma_i$, $\tgamma_j$, $\teta_i$ and $\teta_j$ is shown as in \Cref{fig:p-s}. Then
$$\ind_{M}(\talpha_{i},\talpha_{j}) + \ind_{Y}(\tbeta_{j},\tbeta_{i}) = 1.$$
 \end{lemma}

\begin{proof}
This follows directly from  $\ind_{Z_i}(\talpha_{i}, \tbeta_{i})=0=\ind_{Z_j}(\talpha_{j}, \tbeta_{j})$, and that the grading of the polygon cut by $\talpha_{i}$, $\talpha_{j}$, $\tbeta_{i}$ and $\tbeta_{j}$ is (homotopic to) a constant. 

\begin{figure}[htpb]
	\begin{tikzpicture}[scale=.6]
	\draw[blue,thick] (-4,-2)to(0,2)\bpt to(4,-2);
	\draw[red,thick] (-4,2)to(0,-2)\rpt to(4,2);
	\draw[blue] (-1.25,.75)node[above]{$\tgamma_i$} (1.25,.75)node[above]{$\tgamma_j$};
	\draw[red] (-1.25,-.75)node[below]{$\teta_i$} (1.25,-.75)node[below]{$\teta_j$};
	\draw[blue]
	 (0,2)node[above]{$M$};\draw[red](0,-2)node[below]{$Y$};
\draw	(-2.3,0)node[left]{$Z_i$} (2.3,0)node[right]{$Z_j$};
	\end{tikzpicture}
	\caption{Intersection indices.}\label{fig:p-s}
\end{figure}
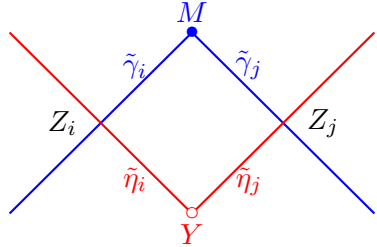

\end{proof}



\begin{definition}
A graded quiver $Q$ with relation $I$ is called \emph{gentle} if the following hold.
\begin{itemize}
	\item The relation set $I \subseteq\left\{ab \mid a, b \in Q_{1}, t(a)=s(b)\right\}$.
	\item For every vertex $i \in Q_{0}$, there are at most two arrows entering and at most two arrows leaving $i$.
	\item For every arrow $a \in Q_{1}$, there is at most one arrow $b \in Q_{1}$ (or $b' \in Q_{1}$) such that $ab \in I$ (or $b'a \in I$).
	\item For every arrow $a \in Q_{1}$, there is at most one arrow $b \in Q_{1}$ (or $b' \in Q_{1}$) such that $ab \notin I$ (or $b'a \notin I$).
\end{itemize}
A \emph{graded gentle algebra} is a finite-dimensional graded algebra $\Lambda=\k Q/\langle I\rangle$ for a gentle grade quiver $Q$ with relation $I$.
\end{definition}


We introduce the quadratic duality for graded gentle algebras.

\begin{definition}\label{def:quad dual}
	Let $\Lambda=\k Q/\langle I \rangle$ be a graded gentle algebra. The \emph{quadratic duality} (see \cite{quadratic algebra,BGS96}) of $\Lambda$ is defined as the graded algebra $\Lambda^{!}=\k Q^{\op}/\langle I^{\perp} \rangle $, where
\begin{itemize}
	\item $Q^{\op}_{0}=Q_0$,
	\item there is an arrow $a^{\ast}:j\to i$ in $Q^{\op}_{1}$ whenever there is an arrow $a:i\to j$ in $Q_{1}$,
	\item the degree of the arrow $a^{\ast}$ is $|a^\ast|=1-|a|$,
	\item $I^{\perp}=\{b^{\ast}a^{\ast}|a,b \in Q_{1},\ t(a)=s(b),\ ab \notin I \}$.
\end{itemize}
\end{definition}
It is not hard to see that the quadratic duality $\Lambda^{!}$ of a gentle algebra $\Lambda$ is also gentle. It is shown in \cite[Prop. 1.25]{OPS} that a gentle algebra and its quadratic duality can be obtained by dual full formal arc systems. Combining this with \Cref{lem:deg}, we can get the following graded version.

\begin{proposition}\label{prop:quad dual}
    Let $\mfS^{\lambda}$ be a graded marked surface,  $\mfU$ a full formal open arc system and $\mfV$ its dual. Then we have isomorphisms of graded algebras
$$(\Lambda_{\mfU})^!\cong \Lambda_{\mfV}\ \text{and}\ 
	(\Lambda_{\mfV})^!\cong \Lambda_{\mfU}.
$$
\end{proposition}

\section{A geometric model for perfect derived categories}\label{sec:geo der eq}
From now on, we regard graded algebras as differential graded algebras with zero differentials.

\subsection{DG algebras and their derived categories}

Let us first review some preliminaries on differential graded(=dg) algebras. We follow \cite{Ke06}. Let $\Lambda$ be a dg $\k$-algebra. When we mention a dg $\Lambda$-module, we always mean a dg right $\Lambda$-module. 

For any two dg $\Lambda$-modules $M$ and $N$, the \emph{morphism
complex} $\mcHom_{\Lambda}(M,N)$ is the graded $\k$-vector space whose $i$-th component $\mcHom_{\Lambda}^i(M,N)$ is the subspace of
the product $\prod_{j\in\mbZ}\Hom_{\k}(M^j,N^{j+i})$ consisting
of morphisms $f$ such that $f(ma)=f(m)a$, for all $m\in M$ and
all $a\in \Lambda$, together with the differential $d$, given by
$$d(f)=f\circ d_M-(-1)^{|f|}d_N\circ f,$$
where $|f|$ is the degree of a homogeneous morphism $f$. Denote by $\mathcal{C}_{dg}(\Lambda)$ the dg category whose objects are dg $\Lambda$-modules, and whose morphism spaces are morphism complexes between dg $\Lambda$-modules. 

The category $\mathcal{C}(\Lambda)$ of dg $\Lambda$-modules is the category
whose objects are dg $\Lambda$-modules, and whose morphism spaces $\Hom_{\mathcal{C}(\Lambda)}(M,N)=Z^0\mcHom_{\Lambda}(M,N)$. The morphisms in $\mathcal{C}(\Lambda)$ are called \emph{dg $\Lambda$-module morphisms}. 
The \emph{homotopy category} $\mathcal{H}(\Lambda)$ of dg
$\Lambda$-modules is the category whose
objects are the dg $\Lambda$-modules, and whose morphism spaces $\Hom_{\mathcal{H}(\Lambda)}(M,N)=H^0\mcHom_{\Lambda}(M,N)$. A dg $\Lambda$-module morphism is called a \emph{homotopy equivalence} if it becomes invertible in $\mathcal{H}(\Lambda)$. The category $\mcH(\Lambda)$ is a triangulated category whose shift
functor is the shift of dg $\Lambda$-modules $M \mapsto M[1]$, and we denote this shift functor also by $[1]$. The \emph{derived
category} $\mcD(\Lambda)$ of dg $\Lambda$-modules is the localization of
$\mcH(\Lambda)$ at the full subcategory of acyclic dg $\Lambda$-modules.

A dg $\Lambda$-module $M$ is said to be \emph{$\oK$-projective} if $\mcHom_{\Lambda}(M,N)$ is acyclic for any acyclic dg $\Lambda$-module $N$. If $M$ is $\oK$-projective, then there is a canonical isomorphism for any dg $\Lambda$-module $N$,
\[
\Hom_{\mcD(\Lambda)}(M,N)\cong H^0\mcHom_{\Lambda}(M,N)=\Hom_{\mcH(\Lambda)}(M,N).
\]
By \cite[Sec. 3.1]{Ke94}, for any dg $\Lambda$-module $M$, there is a quasi-isomorphism $\mfp M\to M$ of dg $\Lambda$-modules with $\mfp M$ being $\oK$-projective. We call $\mfp  M$ a
\emph{$\oK$-projective resolution} of $M$.

The \emph{perfect derived category}
$\per \Lambda$ is the smallest full
subcategory of $\mcD(\Lambda)$ containing $\Lambda$ which is stable under
taking shifts, extensions and direct summands. 


Let $\Lambda_1$ and $\Lambda_2$ be two dg algebras and $\varphi: \Lambda_1 \rightarrow \Lambda_2$ a dg algebra homomorphism. The restriction functor 
$$\rest{\varphi}:\mathcal{C}_{dg}(\Lambda_2)\to \mathcal{C}_{dg}(\Lambda_1)$$ 
is a dg functor, giving each dg $\Lambda_2$-module $M$ a dg $\Lambda_1$-module structure $M_{\Lambda_1}$ along $\varphi$. 
Since $\rest{\varphi}$ sends acyclic dg $\Lambda_2$-modules to acyclic dg $\Lambda_1$-modules, there is an induced triangle functor 
$$\rest{\varphi}:\mcD(\Lambda_2)\to \mcD(\Lambda_1).$$

A dg algebra homomorphism is called a dg algebra \emph{quasi-isomorphism} if it is a quasi-isomorphism of complexes.


\begin{theorem}[\cite{Ke94}, Sec. 7.1]\label{thm:quasi-iso}
Let $\varphi: \Lambda_1 \rightarrow \Lambda_2$ be a dg algebra quasi-isomorphism. Then $\varphi:(\Lambda_1)_{\Lambda_1}\to (\Lambda_2)_{\Lambda_1}=\rest{\varphi}((\Lambda_2)_{\Lambda_2})$ is a dg $\Lambda_1$-module quasi-isomorphism. Moreover, the triangle functor $\rest{\varphi}:\mcD(\Lambda_2) \rightarrow \mcD(\Lambda_1)$ is an equivalence, which restricts to a triangle equivalence
$$
\per \Lambda_2 \xrightarrow{\simeq} \per \Lambda_1.
$$

\end{theorem}
The following theorem is known as the derived Morita theorem.
\begin{theorem}[\cite{Ke94}, Sec. 7.3]\label{thm:keller}
Let $\Lambda$ be a dg algebra. Let $M \in \mcD(\Lambda)$ and $\mfp M$ a $\oK$-projective resolution of $M$.
Then there is a triangle equivalence
$$\mcHom_{\Lambda}(\mfp M,-):\thick_{\mcD(\Lambda)}(M) \xrightarrow{\simeq} \per\mcHom_{\Lambda}(\mfp M,\mfp M),$$
where $\thick_{\mathcal{D}(\Lambda)}(M)$ is the smallest thick subcategory of $\mathcal{D}(\Lambda)$ that contains $M$.
\end{theorem}
\begin{remark}
The dg algebra $\mfR\Hom_{\mcD(\Lambda)}(M,M):=\mcHom_{\Lambda}(\mfp M,\mfp M)$ is called the \emph{derived endomorphism algebra} of $M$, which does not depend on the choice of the $\oK$-projective resolution $\mfp M$ of $M$, up to quasi-isomorphism, cf. \cite{Ke06}.
\end{remark}

\subsection{A geometric model for perfect derived categories}\label{sec:geo per}\

Recall that we have fixed a full formal open arc system $\mfU=\{\talpha_1,\cdots,\talpha_n\}$, and its dual $\mfV=\{\tbeta_1,\cdots,\tbeta_n\}$. For simplicity, we denote $\Lambda_{\mfU}$ by $\Lambda$. 

There is a complete set of orthogonal primitive idempotents of $\Lambda$: $\{e_1,\cdots,e_n\}$, where $e_i$ corresponds to the graded simple open arc $\tgamma_i$. We denote the direct summand $e_i\Lambda$ of $\Lambda$ by $P_i$, which is an indecomposable dg $\Lambda$-module. Then we have $\Lambda=\oplus_{i=1}^nP_i$. For any path $p$ from $i$ to $j$ in $Q_{\mfU}$ such that $p$ is not in the ideal generated by the relation set $I_{\mfU}$, there is an associated non-zero dg $\Lambda$-module morphism from $P_j$ to $P_i[|p|]$, which is denoted by $\phi_p$.

\begin{notations}\label{not:arc}
Let $\tsigma$ be a graded open arc with a fixed orientation. 
We denote by (cf. \Cref{the sign of delta})
\begin{itemize}
    \item $\tbeta_{\K{\tsigma}{1}},\cdots,\tbeta_{\K{\tsigma}{r_{\tsigma}}}$ the arcs in $\mfU^*$ that $\tsigma$ intersects in order,
    \item $\ver{\tsigma}{i}$ the intersection between $\tsigma$ and  $\tbeta_{\K{\tsigma}{i}}$,
    \item $\Rho{\tsigma}{i}=\ind_{\ver{\tsigma}{i}}(\tsigma,\tbeta_{\K{\tsigma}{i}})$,
    \item $\ver{\tsigma}{0}, \ver{\tsigma}{r_{\tsigma}+1} \in \mfM$ the endpoints of $\tsigma$,
    \item $\tsigma(i,i+1)$ the arc segment between $\ver{\tsigma}{i}$ and $\ver{\tsigma}{i+1}$ for $1 \leq i \leq r_{\tsigma}$,
    \item $\poly{\tsigma}{i}$ the $\mfU^*$-polygon containing the arc segment $\tsigma(i,i+1)$,
    \item $\M{\tsigma}{i}$ the unique open marked point contained in $\poly{\tsigma}{i}$, and
    \item $\paths{\tsigma}{i}$ the path between $\K{\tsigma}{i}$ and $\K{\tsigma}{i+1}$ in $Q_{\mfU}$, corresponding to the segment $\tsigma(i,i+1)$.
\end{itemize}
\end{notations}

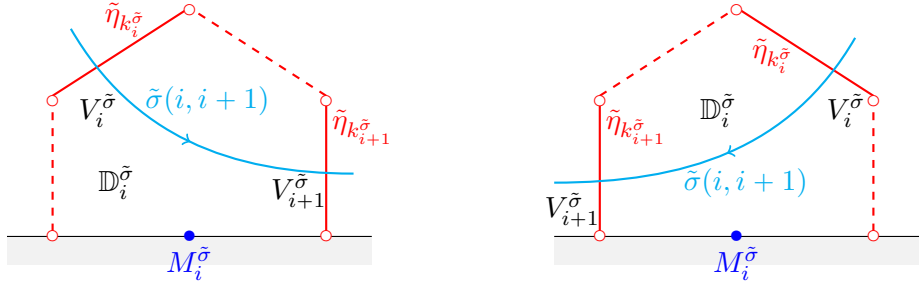
\begin{figure}[htpb]
	\begin{tikzpicture}[scale=1.2]
\draw[thick]
(-4,0) to (0,0);
\draw[dashed,red,thick]
(-3.5,0.06) to (-3.5,1.45);
\draw[red,thick]
(-3.5,1.56) to (-2.05,2.5);
\draw[red,thick]
(-.5,0.06) to (-.5,1.45);
\draw[dashed,red,thick]
(-.5,1.56) to (-1.95,2.5);
\draw[cyan,thick,->-=0.5]
(-3.3,2.3) to
[out=-60,in=180] (-.2,.7);
\draw[thick]
(-3,1.4)node{$\ver{\tsigma}{i}$}
(-.8,0.5)node{$\ver{\tsigma}{i+1}$};
\draw[thick,red]
(-2.7,2.4)node{$\teta_{\K{\tsigma}{i}}$}
(-.1,1.2)node{$\teta_{\K{\tsigma}{i+1}}$};
\draw[thick,cyan]
(-1.8,1.5)node{$\tsigma(i,i+1)$};
\fill[fill=gray!10]
(-4,0) rectangle (-0,-0.3);
\draw[blue]
(-2,0)node[below]{$\M{\tsigma}{i}$};
\draw
(-2.8,.6)node{$\poly{\tsigma}{i}$};
\draw
(-2,0)\bpt
(-.5,0)\rpt
(-3.5,0)\rpt
(-3.5,1.5)\rpt
(-2,2.5)\rpt
(-.5,1.5)\rpt;
\begin{scope}[shift={(6,0)}]
\draw[thick]
(-4,0) to (0,0);
\draw[red,thick]
(-3.5,0.06) to (-3.5,1.45);
\draw[dashed,red,thick]
(-3.5,1.56) to (-2.05,2.5);
\draw[dashed,red,thick]
(-.5,0.06) to (-.5,1.45);
\draw[red,thick]
(-.5,1.56) to (-1.95,2.5);
\draw[cyan,thick,->-=0.5]
(-.7,2.2) to [out=-120,in=0] (-4,.6);
\draw[red,thick]
(-3.1,1.2)node{$\teta_{\K{\tsigma}{i+1}}$}
(-1.6,2)node{$\teta_{\K{\tsigma}{i}}$};
\draw[thick]
(-3.8,0.3)node{$\ver{\tsigma}{i+1}$}
(-.8,1.4)node{$\ver{\tsigma}{i}$};
\draw[cyan,thick]
(-1.9,.6)node{$\tsigma(i,i+1)$};
\fill[fill=gray!10]
(-4,0) rectangle (-0,-0.3);
\draw
(-2,0)\bpt
(-.5,0)\rpt
(-3.5,0)\rpt
(-3.5,1.5)\rpt
(-2,2.5)\rpt
(-.5,1.5)\rpt;
\draw[blue]
(-2,0)node[below]{$\M{\tsigma}{i}$};
\draw
(-2.2,1.4)node{$\poly{\tsigma}{i}$};
\end{scope}
	\end{tikzpicture}
	\caption{Notations for a graded open arc $\tsigma$.}
	\label{the sign of delta}
\end{figure}

Using \Cref{not:arc}, for any graded open arc $\tsigma$, there is an associated sequence of paths in $Q_{\mfU}$
\[
\xymatrix@C=3pc{
  \K{\tsigma}{1}\ar@{-}[r]^{\paths{\tsigma}{1}}&\K{\tsigma}{2}\ar@{-}[r]^{\paths{\tsigma}{2}}&\cdots
  \ar@{-}[r]^{\paths{\tsigma}{r_{\tsigma}-1}}&\K{\tsigma}{r_{\tsigma}}},
  \]
  where each path $\paths{\tsigma}{i}$ is not in the ideal generated by the relation set $I_{\mfU}$.

\begin{construction}\label{cons:obj}


A dg $\Lambda$-module $\oX(\tsigma)=(|\oX(\tsigma)|,d_{\tsigma})$ associated to a graded open arc $\tsigma$ is constructed as follows.
\begin{itemize}
	\item The underlying graded module $|\oX(\tsigma)|=\bigoplus^{r}_{i=1} P_{\K{\tsigma}{i}}{[\Rho{\tsigma}{i}]}$.
    \item 
    The differential $d_{\tsigma}$ consists of the nonzero components $\phi^{\tsigma}_i$ of degree 1 between $P_{\K{\tsigma}{i+1}}{[\Rho{\tsigma}{i+1}]}$ and $P_{\K{\tsigma}{i}}{[\Rho{\tsigma}{i}]}$, induced by $\phi_{\paths{\tsigma}{i}}$, where
    \begin{itemize}
        \item $\phi^{\tsigma}_i$ is from $P_{\K{\tsigma}{i+1}}{[\Rho{\tsigma}{i+1}]}$ to $P_{\K{\tsigma}{i}}{[\Rho{\tsigma}{i}]}$, if $\M{\tsigma}{i}$ is to the left of $\tsigma(i,i+1)$ in $\poly{\tsigma}{i}$ (see the right picture of \Cref{the sign of delta}), and
        \item $\phi^{\tsigma}_i$ is from $P_{\K{\tsigma}{i}}{[\Rho{\tsigma}{i}]}$ to $P_{\K{\tsigma}{i+1}}{[\Rho{\tsigma}{i+1}]}$, if $\M{\tsigma}{i}$ is to the right of $\tsigma(i,i+1)$ in $\poly{\tsigma}{i}$ (see the left picture of \Cref{the sign of delta}).
    \end{itemize}
\end{itemize}
\end{construction}

The dg $\Lambda$-module $\oX(\tsigma)$ can be written simply as
$$\oX(\tsigma) =  {\xymatrix@C=2.5pc{P_{\K{\tsigma}{1}}{[\Rho{\tsigma}{1}]} \ar@{-}[r]^(0.6){\phi^{\tsigma}_1} & \cdots \ar@{-}[r]^(0.4){\phi^{\tsigma}_{{r_{\tsigma}-1}}} & P_{\K{\tsigma}{r_{\tsigma}}}{[\Rho{\tsigma}{r_{\tsigma}}]} \notag
}}.$$

\begin{theorem}[\cite{QZZ}, Thm. 4.11, Thm. 6.1]\label{thm:QZZ}
    The map $\oX$ constructed above gives rise to an injective map from the set $\OC$ to the set of isoclasses of indecomposable objects in $\per\Lambda$. Moreover, the following hold. 
    \begin{itemize}
        \item For any $\tsigma\in\OC$ and any $\rho\in\mathbb{Z}$, we have the isomorphism 
        \begin{equation}\label{eq:shift}
            \oX(\tsigma[\rho])\cong \oX(\tsigma)[\rho].
        \end{equation}
        \item For any two $\tsigma,\tsigma'\in\OC$ and $\rho\in\mathbb{Z}$,  there is an associated $f_{\rl}\in Z^\rho\mcHom_{\Lambda}(M,N)$ to $\rl\in\overrightarrow{\cap}^{\rho}(\tsigma,\tsigma')$, such that $\bar{f}_{\rl}$'s form a basis of $H^\rho\mcHom_{\Lambda}(M,N)$, where $\bar{f}_{\rl}$ is the homotopy class of $f_{\rl}$. In particular, we have
    \begin{equation}\label{eq:int=dim}
        \Int^{\rho}(\tsigma,\tsigma')=\dim\Hom_{\per\Lambda}(\oX(\tsigma),\oX(\tsigma')[\rho]).
    \end{equation}
    \end{itemize}
\end{theorem}

We recall the construction of the morphism $f_{\rl}$ in \Cref{thm:QZZ} for the case that $\rl$ is at a marked point on the boundary, which will be used later.

\begin{construction}\label{cons:mor}

Let $\rl\in\overrightarrow{\cap}^{\rho}(\tsigma,\tsigma')$ at a marked point $M\in\mfM$. Using Notations~\ref{not:arc} for $\tsigma$ and $\tsigma'$, to simplify the notations, we denote $s=r_{\tsigma}$ and $t=r_{\tsigma'}$;  $m_u=\Rho{\tsigma}{u}$ and $n_v=\Rho{\tsigma'}{v}$; $i_u=\K{\tsigma}{u}$ and $j_v=\K{\tsigma'}{v}$; $\phi_u=\phi^{\tsigma}_u$ and $\phi'_v=\phi^{\tsigma'}_v$. Then the dg $\Lambda$-module $\oX(\tsigma)$ is written as 
\begin{gather}\notag
\xymatrix@C=3pc{ P_{i_1}{[m_1]} \ar@{-}[r]^{\phi_1} & P_{i_2}{[m_2]} \ar@{-}[r]^(0.65){\phi_2} & \cdots  \ar@{-}[r]^(0.4){\phi_{s-1}} & P_{i_s}{[m_s]}
},
\end{gather}
and the dg $\Lambda$-module $\oX(\tsigma')$ is written as
\begin{gather}\notag
\xymatrix@C=3pc{ P_{j_1}{[n_1]} \ar@{-}[r]^{\phi_1'} & P_{j_2}{[n_2]} \ar@{-}[r]^(0.55){\phi_2'} & \cdots  \ar@{-}[r]^(0.45){\phi'_{{t-1}}} & P_{j_t}{[n_t]}
}.
\end{gather}
Let $k$ be the maximum such that $i_u=j_u$ for all $1\leq u\leq k$. Then the non-zero components of $f_\alpha$ are $(-1)^{\rho(u-1)}\id:P_{i_u}[m_u]\to P_{j_u}[n_u]$, $1\leq u\leq k$, and when the path $p$ induced by the arc segment between $\tbeta_{i_{k+1}}$ and $\tbeta_{j_{k+1}}$ (if exists) in the $\mfU^*$-polygon $\poly{\tsigma}{k+1}=\poly{\tsigma'}{k+1}$ is from $j_u$ to $i_u$, $(-1)^{\rho k}\phi$, where $\phi=\phi_p[m_{k+1}]:P_{i_{k+1}}[m_{k+1}]\to P_{j_{k+1}}[n_{k+1}]$. More precisely, there are the following cases.
\begin{enumerate}
\item Case that ${i_1} \neq {j_1}$.
There is a morphism $\phi=\phi_p[m_1]:P_{i_1}{[m_1]}\to P_{j_1}{[n_1]}$ in $\mcC_{dg}(\Lambda)$ of degree $\rho$, where $p$ is the path in $Q_{\mfU}$ from $j_1$ to $i_1$, induced by the arc segment between $\tbeta_{i_1}$ and $\tbeta_{j_1}$ in the $\mfU^*$-polygon $\poly{\tsigma}{1}=\poly{\tsigma'}{1}$, see the middle picture in \Cref{fig:mor123}. In this case, $f_{\rl}$ is given by the following diagram
\begin{gather} \label{eq:mor1}
\xymatrix@C=3pc{P_{i_1}{[m_1]}
\ar[d]_(0.4){\phi}
 \ar@{-}[r]^{\phi_1} & P_{i_2}{[m_2]} \ar@{-}[r]^(0.6){\phi_2} & \cdots  \ar@{-}[r]^(0.4){\phi_{{s-1}}} & P_{i_s}{[m_s]} \\
P_{j_1}{[n_1]} \ar@{-}[r]^{\phi'_1} & P_{j_2}{[n_2]} \ar@{-}[r]^{\phi'_2} & \cdots  \ar@{-}[r]^{\phi'_{{t-1}}} & P_{j_t}{[n_t]}
.}
\end{gather}
 \item Case that $s\leq t$ and $i_u=j_u$ for all $1 \leq u \leq s$, see the left picture in \Cref{fig:mor123}. In this case,
 $f_{\rl}$ is given by the following diagram
    \begin{gather}  \label{eq:mor2}
        \xymatrix@C=3pc{ P_{i_1}{[m_1]} \ar[d]^{\id} \ar@{-}[r]^(0.6){\phi_1} &  \cdots  \ar[d]^{\cdots} \ar@{-}[r]^(0.4){\phi_{k-1}} & P_{i_k}{[m_k]}  \ar[d]^{(-1)^{\rho(k-1)}\id}\\
    P_{j_1}{[n_1]} \ar@{-}[r]^{\phi'_1}  & \cdots  \ar@{-}[r]^{\phi'_2} & P_{j_k}{[n_k]} \ar@{<-}[r]^{\phi'_k}   & \cdots \ar@{-}[r]^{\phi'_{t-1}} & P_{j_t}{[n_t]}
.}
    \end{gather}
\item Case that $s>t$ and $i_u=j_u$ for all $1 \leq u \leq t$, see the right picture in \Cref{fig:mor123}. In this case,
 $f_{\rl}$ is given by the following diagram
 \begin{gather}  \label{eq:mor3}
        \xymatrix@C=3pc{ P_{i_1}{[m_1]} \ar[d]^{\id} \ar@{-}[r]^(0.6){\phi_1}  & \cdots  \ar[d]^{\cdots} \ar@{-}[r]^(0.4){\phi_{k-1}} & P_{i_k}{[m_k]}  \ar[d]^{(-1)^{\rho(k-1)}\id} \ar[r]^(0.6){\phi_k}  & \cdots \ar@{-}[r]^(0.4){\phi_s} & P_{i_s}{[m_s]}  \\
    P_{j_1}{[n_1]} \ar@{-}[r]^{\phi'_1} & \cdots  \ar@{-}[r]^{\phi'_{k-1}} & P_{j_k}{[n_k]}    .}
\end{gather}

\begin{figure}[htpb]
\begin{tikzpicture}
\draw[thick]
(-3.5,0) to (0,0)
(-3.5,4.3) to (0,4.3);
\fill[fill=gray!10]
(-3.5,0) rectangle (0,-0.3);
\draw[dashed,red,thick]
(-3.25,0.1) to (-3.8,1.25)
(-0.3,0.1) to (0.2,1.25);
\draw[red,thick]
(-3.75,1.3) to (0.15,1.3);
\draw[red,thick]
(-3.75,2.5) to (0.14,2.5);
\draw[dashed,red,thick]
(-3.8,2.6) to (-3.8,3.46)
(0.2,2.6) to (0.2,3.4)
(-3.8,3.6) to (-3,4.3);
\draw[red,thick]
(0.2,3.58) to (-0.6,4.25);
\fill[fill=gray!10]
(-3.5,4.6) rectangle (0,4.3);
\draw[cyan,thick]
(-1.85,4.3) to [out=180,in=90] (-3,2) to[out=-90,in=160] (-1.7,0)
(-1.7,0) to [out=20,in=-100] (0.1,4.1);
\draw[cyan]
(-3.3,2)node{$\tsigma$}
(0.3,2)node{$\tsigma'$};
\draw[red]
(-1.7,1.5)node{$\teta_{i_1}=\teta_{j_1}$}
(-1.7,2.8)node{$\teta_{i_k}=\teta_{j_k}$}
(-0.45,3.8)node{$\teta_{j_{k+1}}$}
(-1.7,2.2)node{$\vdots$};
\draw[thick,Green,->-=0.55]
(-2.45,0.5) to [out=50] (-0.9,0.5);
\draw[Green](-1.6,1.1)node{${\rl}$};
\draw(-3.2,0)\rpt
(-1.7,0)\bpt
(-1.85,4.3)\bpt
(-0.3,0)\rpt
(-3.8,2.5)\rpt
(-3.8,3.5)\rpt
(0.2,2.5)\rpt
(0.2,3.5)\rpt
(-3.8,1.3)\rpt
(0.2,1.3)\rpt
(-3,4.3)\rpt
(-0.6,4.3)\rpt;
\draw[blue]
(-1.7,0)node[below]{$M$};
\begin{scope}[shift={(5,0)}]
\draw[thick]
(-3.5,0) to (0,0);
\draw[dashed,red,thick]
(-3.25,0.1) to (-3.8,1.25)
(-0.3,0.1) to (0.2,1.26);
\draw[red,thick]
(-3.83,1.37) to (-3.2,2.45)
(0.2,1.37) to (-0.4,2.45);
\draw[dashed,red,thick]
(-3.1,2.5) to (-0.4,2.5);
\draw[cyan,thick]
(-1.7,0) to [out=150,in=-60] (-4,2.5)
(-1.7,0) to [out=30,in=-120] (0.3,2.5);
\draw[cyan]
(-2.8,0.6)node{$\tsigma$}
(-0.6,0.6)node{$\tsigma'$};
\draw[red]
(-3.1,2)node{$\teta_{i_1}$}
(-0.4,2)node{$\teta_{j_1}$};
\fill[fill=gray!10]
(-3.5,0) rectangle (0,-0.3);
\draw[thick,Green,->-=0.55]
(-2.4,0.5) to [out=50]
(-1,0.5);
\draw[Green](-1.7,1.1)node{${\rl}$};
\draw[blue]
(-1.7,0)node[below]{$M$};
\draw(-3.2,0)\rpt
(-1.7,0)\bpt
(-0.3,0)\rpt
(-3.8,1.3)\rpt
(-3.2,2.5)\rpt
(0.2,1.3)\rpt
(-0.4,2.5)\rpt;
	\end{scope}
\begin{scope}[shift={(10,0)}]
	\draw[thick]
(-3.5,0) to (0,0)
(-3.5,4.3) to (0,4.3);
\fill[fill=gray!10]
(-3.5,0) rectangle (0,-0.3);
\draw[dashed,red,thick]
(-3.25,0.1) to (-3.8,1.25)
(-0.3,0.1) to (0.2,1.25);
\draw[red,thick]
(-3.75,1.3) to (0.15,1.3);
\draw[red,thick]
(-3.75,2.5) to (0.14,2.5)
(-3.8,3.6) to (-3,4.3);
\draw[dashed,red,thick]
(-3.8,2.6) to (-3.8,3.46)
(0.2,2.6) to (0.2,3.4);
\draw[dashed,red,thick]
(0.2,3.58) to (-0.6,4.3);
\fill[fill=gray!10]
(-3.5,4.6) rectangle (0,4.3);
\draw[cyan,thick]
(-1.85,4.3) to [out=-10,in=90] (-.3,2) to [out=-90,in=10] (-1.7,0)
(-1.7,0) to [out=150,in=-100] (-3.5,4.1);
\draw[cyan]
(-3,2)node{$\tsigma$}
(0,2)node{$\tsigma'$};
\draw[red]
(-1.7,1.5)node{$\teta_{i_1}=\teta_{j_1}$}
(-1.7,2.8)node{$\teta_{i_k}=\teta_{j_k}$}
(-2.8,3.8)node{$\teta_{i_{k+1}}$}
(-1.7,2.2)node{$\vdots$};
\draw[thick,Green,->-=0.55]
(-2.4,0.5) to [out=50] (-.8,0.5);
\draw[Green](-1.6,1.1)node{${\rl}$};
\draw(-3.2,0)\rpt
(-1.7,0)\bpt
(-1.85,4.3)\bpt
(-0.3,0)\rpt
(-3.8,2.5)\rpt
(-3.8,3.5)\rpt
(0.2,2.5)\rpt
(0.2,3.5)\rpt
(-3.8,1.3)\rpt
(0.2,1.3)\rpt
(-3,4.3)\rpt
(-0.6,4.3)\rpt;
\draw[blue]
(-1.7,0)node[below]{$M$};
\end{scope}
\end{tikzpicture}
 \caption{Morphisms from oriented intersections-1.}
 \label{fig:mor123}	
\end{figure}
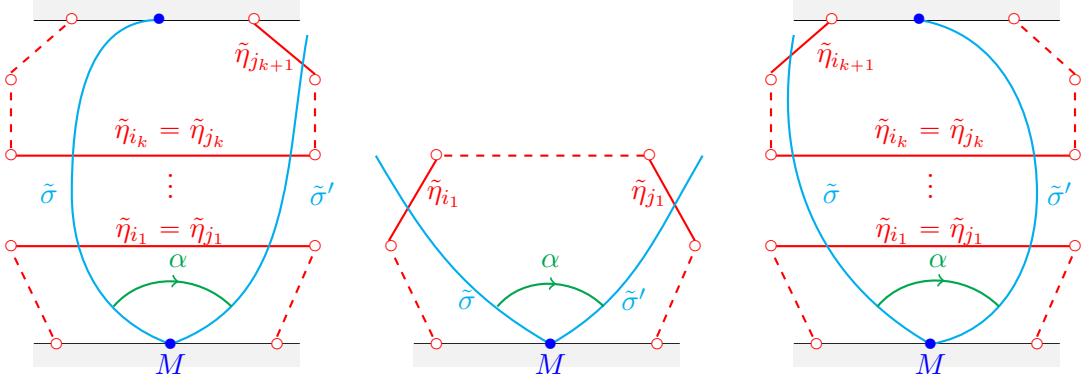

  \item Case that $i_u=j_u$ for all $1 \leq u \leq k$ and $i_{k+1}\neq j_{k+1}$, where $k < \min\{s,t\}$. In this case, the simple open arc $\talpha_{i_k}$ divides the polygon $\poly{\tsigma}{k}=\poly{\tsigma'}{k}$ into two parts, hence there are two cases.
 \begin{itemize}
 	\item  Case that $\tbeta_{i_{k+1}}$ and $\tbeta_{j_{k+1}}$ are in the different parts, see the middle picture in \Cref{fig:mor45}. In this case, $f_{{\rl}}$ is given by the following diagram
\begin{gather}  \label{eq:mor4}
        \xymatrix@C=3pc{ P_{i_1}{[m_1]} \ar[d]^{\id} \ar@{-}[r]^(0.6){\phi_1} & \cdots  \ar[d]^{\cdots} \ar@{-}[r]^(0.4){\phi_{k-1}} & P_{i_{k}}{[m_k]}  \ar[d]^{(-1)^{\rho(k-1)}\id} \ar[r]^(0.6){\phi_{k}} & \cdots \ar@{-}[r]^(0.4){\phi_{s-1}} & P_{i_s}{[m_s]}  \\
    P_{j_1}{[n_1]} \ar@{-}[r]^{\phi'_1} & \cdots  \ar@{-}[r]^(0.4){\phi'_{k-1}} & P_{j_{k}}{[n_{k}]}   &  \cdots
    \ar[l]_(0.4){\phi'_{k}}
     \ar@{-}[r]^{\phi'_{t-1}} & P_{j_t}{[n_t]}
.}
\end{gather}

\item Case that $\tbeta_{i_{k+1}}$ and $\tbeta_{j_{k+1}}$ are in the same part. There is a morphism $\phi=\phi_p[m_{k+1}]: P_{i_{k+1}}{[m_{k+1}]}\to P_{j_{k+1}}{[n_{k+1}]}$ of degree $\rho$ in $\mcC_{dg}(\Lambda)$, where $p$ is the path from $j_{k+1}$ to $i_{k+1}$ in $Q_{\mfU}$, induced by the arc segment between $\tbeta_{i_{k+1}}$ and $\tbeta_{j_{k+1}}$ in the $\mfU^*$-polygon $\poly{\tsigma}{k}=\poly{\tsigma'}{k}$, see the left and the right pictures in \Cref{fig:mor45}. In this case, $f_{\rl}$ is given by the following diagram
\begin{gather}  \label{eq:mor5}
\xymatrix@C=2.5pc{ P_{i_1}{[m_1]} \ar[d]^{\id} \ar@{-}[r]^(0.6){\phi_1} &  \cdots \ar@{-}[r]^(0.4){\phi_{k-1}} \ar[d]^{\cdots} & P_{i_{k}}{[m_{k}]} \ar[d]^{(-1)^{\rho(k-1)}\id} \ar@{-}[r]^(0.45){\phi_k} & P_{i_{k+1}}[m_{k+1}] \ar[d]^{(-1)^{\rho k}\phi} \ar@{-}[r]^(0.65){\phi_{k+1}} & \cdots \ar@{-}[r]^(0.4){\phi_{s-1}} & P_{i_s}{[m_s]}
\\
    P_{j_1}{[n_1]} \ar@{-}[r]^(0.6){\phi'_1} & \cdots  \ar@{-}[r]^(0.35){\phi'_{k-1}}    & P_{j_{k}}{[n_{k}]}  \ar@{-}[r]^(0.45){\phi'_k} & P_{j_{k+1}}[n_{k+1}] \ar@{-}[r]^(0.65){\phi'_{k+1}} & \cdots \ar@{-}[r]^{\phi'_{t-1}} & P_{j_t}{[n_t]}
.}
\end{gather}
\end{itemize}

\begin{figure}[htpb]
\begin{tikzpicture}
\draw[thick]
(-3.5,0) to (0,0)
(-3.5,4.6) to (0,4.6);
\fill[fill=gray!10]
(-3.5,0) rectangle (0,-0.3);
\draw[dashed,red,thick]
(-2.85,0.1) to (-3.4,1.26)
(-0.65,0.1) to (-0.2,1.27);
\draw[red,thick]
(-3.35,1.3) to (-0.26,1.3);
\draw[red,thick]
(-3.35,2.3) to (-0.26,2.3);
\draw[red,thick]
(-3.6,3.08) to (-3.6,3.94);
\draw[dashed,red,thick]
(0,3.08) to (0,3.94);
\draw[dashed,red,thick]
(-3.6,2.95) to (-3.44,2.36)
(0,2.9) to (-0.2,2.3);
\draw[red,thick]
(-3.6,4.08) to (-2.6,4.54);
\draw[dashed,red,thick]
(0,4.08) to (-1,4.54);\fill[fill=gray!10]
(-3.5,4.9) rectangle (0,4.6);
\draw[red]
(-1.7,1.5)node{$\teta_{i_1}=\teta_{j_1}$}
(-2.5,2.6)node{$\teta_{i_k}=\teta_{j_k}$}
(-1.7,2)node{$\vdots$};
\draw[thick,blue]
(-1.75,4.6) to (-1.75,2.2);
\draw[blue]
(-1.4,3.8)node{$\tgamma_{i_k}$};
\draw[cyan,thick]
(-1.7,0) to [out=140,in=-60] (-3.9,3.9)
(-1.7,0) to [out=20,in=-90] (-0.8,1.5) to[out=90,in=-30] (-3.8,4.4);
\draw[cyan]
(-3.3,1.8)node{$\tsigma$}
(-.5,1.8)node{$\tsigma'$};
\draw[thick,Green,->-=0.55]
(-2.3,0.6) to[out=50] (-1,0.6);
\draw[Green](-1.65,0.5)node{${\rl}$};
\draw[red]
(-3.1,3.5)node{$\teta_{i_{k+1}}$}
(-2.5,4.2)node{$\teta_{j_{k+1}}$};
\draw(-2.8,0)\rpt
(-1.7,0)\bpt
(-0.7,0)\rpt
(-1.75,4.6)\bpt
(-2.6,4.6)\rpt
(-1,4.6)\rpt
(-3.4,1.3)\rpt
(-3.4,2.3)\rpt
(-0.2,1.3)\rpt
(-0.2,2.3)\rpt
(-3.6,3)\rpt
(-3.6,4)\rpt
(0,3)\rpt
(0,4)\rpt;
\draw[blue]
(-1.7,0)node[below]{$M$};
\begin{scope}[shift={(5,0)}]
\draw[thick]
(-3.5,0) to (0,0)
(-3.5,4.6) to (0,4.6);
\fill[fill=gray!10]
(-3.5,0) rectangle (0,-0.3);
\draw[dashed,red,thick]
(-2.85,0.1) to (-3.4,1.26)
(-0.65,0.1) to (-0.2,1.27);
\draw[red,thick]
(-3.35,1.3) to (-0.26,1.3);
\draw[red,thick]
(-3.35,2.3) to (-0.26,2.3);
\draw[red,thick]
(-3.6,3.08) to (-3.6,3.94)
(0,3.08) to (0,3.94);
\draw[dashed,red,thick]
(-3.6,2.95) to (-3.44,2.36)
(0,2.9) to (-0.2,2.3);
\draw[dashed,red,thick]
(-3.6,4.08) to (-2.6,4.54)
(0,4.08) to (-1,4.54);
\fill[fill=gray!10]
(-3.5,4.9) rectangle (0,4.6);
\draw[thick,blue]
(-1.75,4.6) to (-1.75,2.23);
\draw[blue]
(-1.4,3.5)node{$\tgamma_{i_k}$};
\draw[red]
(-2,1.5)node{$\teta_{i_1}=\teta_{j_1}$}
(-2.5,2.6)node{$\teta_{i_k}=\teta_{j_k}$}
(-1.8,2.1)node{$\vdots$};
\draw[cyan,thick]
(-1.7,0) to [out=140,in=-70] (-3.8,4)
(-1.7,0) to (0.2,3.8);
\draw[thick,Green,->-=0.55]
(-2.4,0.7) to [out=50] (-1.34,0.7);
\draw[cyan]
(-3.3,1.8)node{$\tsigma$}
(-0.4,1.8)node{$\tsigma'$};
\draw[red]
(-3.1,3.7)node{$\teta_{i_{k+1}}$}
(-0.4,3.7)node{$\teta_{j_{k+1}}$};
\draw[Green]
(-1.9,0.6)node{${\rl}$};
\draw
(-2.8,0)\rpt
(-1.7,0)\bpt
(-0.7,0)\rpt
(-1.75,4.6)\bpt
(-2.6,4.6)\rpt
(-1,4.6)\rpt
(-3.4,1.3)\rpt
(-3.4,2.3)\rpt
(-0.2,1.3)\rpt
(-0.2,2.3)\rpt
(-3.6,3)\rpt
(-3.6,4)\rpt
(0,3)\rpt
(0,4)\rpt;
\draw[blue]
(-1.7,0)node[below]{$M$};
\end{scope}
\begin{scope}[shift={(10,0)}]
\draw[thick]
(-3.5,0) to (0,0)
(-3.5,4.6) to (0,4.6);
\fill[fill=gray!10]
(-3.5,0) rectangle (0,-0.3);
\draw[dashed,red,thick]
(-2.85,0.1) to (-3.4,1.26)
(-0.65,0.1) to (-0.2,1.27);
\draw[red,thick]
(-3.35,1.3) to (-0.26,1.3);
\draw[red,thick]
(-3.35,2.3) to (-0.26,2.3);
\draw[red,thick]
(0,3.08) to (0,3.94);
\draw[dashed,red,thick]
(-3.6,3.08) to (-3.6,3.94)
(-3.6,2.95) to (-3.44,2.36)
(0,2.9) to (-0.2,2.3);
\draw[dashed,red,thick]
(-3.6,4.08) to (-2.6,4.54);
\draw[red,thick]
(0,4.08) to (-1,4.54);
\fill[fill=gray!10]
(-3.5,4.9) rectangle (0,4.6);
\draw[red]
(-1.6,1.5)node{$\teta_{i_1}=\teta_{j_1}$}
(-1,2.6)node{$\teta_{i_k}=\teta_{j_k}$};
\draw[thick,cyan]
(-1.7,0) to [out=50,in=-110](0.2,3.9)
(-1.7,0) to [out=150,in=-90]
(-2.6,1.5) to [out=90,in=-170]
(-0,4.3);	
\draw[thick,Green,->-=0.55]
(-2.35,0.6) to[out=50] (-1.25,0.6);
\draw[Green](-1.8,0.55)node{${\rl}$};
\draw[cyan]
(-2.8,1.9)node{$\tsigma$}
(-0.2,1.9)node{$\tsigma'$};
\draw[red]
(-0.4,3.6)node{$\teta_{j_{k+1}}$}
(-1.1,4.3)node{$\teta_{i_{k+1}}$}
(-1.8,2)node{$\vdots$};
\draw[thick,blue]
(-1.75,4.6) to (-1.75,2.2);
\draw[blue]
(-2,3.8)node{$\tgamma_{i_k}$};
\draw(-2.8,0)\rpt
(-1.7,0)\bpt
(-0.7,0)\rpt
(-1.75,4.6)\bpt
(-2.6,4.6)\rpt
(-1,4.6)\rpt
(-3.4,1.3)\rpt
(-3.4,2.3)\rpt
(-0.2,1.3)\rpt
(-0.2,2.3)\rpt
(-3.6,3)\rpt
(-3.6,4)\rpt
(0,3)\rpt
(0,4)\rpt;
\draw[blue]
(-1.7,0)node[below]{$M$};
\end{scope}
\end{tikzpicture}
 \caption{Morphisms from oriented intersections-2.}
 \label{fig:mor45}
\end{figure}
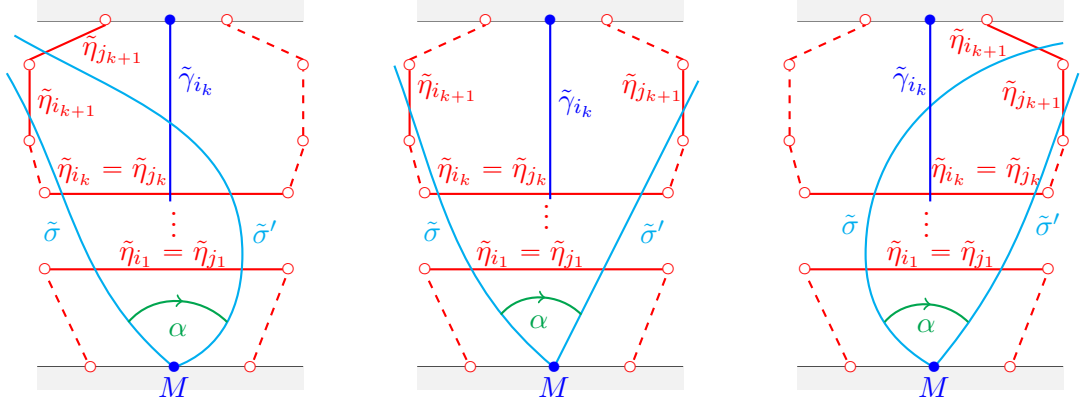

\end{enumerate}
\end{construction}

\begin{notations}
In the case that the degree of $f_{\rl}:\oX(\tsigma)\to \oX(\tsigma')$ is 1 (i.e. $\rho=1$), each of the diagrams \eqref{eq:mor1}-\eqref{eq:mor5} can be viewed as a dg $\Lambda$-module, denoted by $\oX(\tsigma)\xrightarrow{f_{\rl}}\oX(\tsigma')$, or by $\Cone(f_{\rl})$ shortly, in the following way.
\begin{itemize}
\item The underlying graded module is the direct sum of the graded modules appearing in the first and the second rows.
\item The differential is given by all arrows in the diagram.
\end{itemize}
We denote by
\begin{itemize}
    \item $\emb(f_{\rl},X(\tsigma))$ the natural embedding in $\mathcal{C}_{dg}(\Lambda)$ from $X(\tsigma)$ to $\Cone(f_{\rl})$,
    \item $\pro(f_{\rl},X(\tsigma))$ the natural projection in $\mathcal{C}_{dg}(\Lambda)$ from $\Cone(f_{\rl})$ to $X(\tsigma)$,
    \item $\emb(f_{\rl},X(\tsigma'))$ the scale of the natural embedding in $\mathcal{C}_{dg}(\Lambda)$ from $X(\tsigma')$ to $\Cone(f_{\rl})$ by $(-1)^k$,
    \item $\pro(f_{\rl},X(\tsigma'))$ the scale of the natural projection in $\mathcal{C}_{dg}(\Lambda)$ from $\Cone(f_{\rl})$ to $X(\tsigma')$ by $(-1)^k$.
\end{itemize}

    
\end{notations}

\begin{remark}
    If $\Lambda$ is ungraded, then the morphisms in \Cref{cons:mor} are called the quasi-graph maps in \cite{ALP}.
\end{remark}

The following lemma follows directly from \Cref{cons:mor}, via a case-by-case check. Similar results can be found in \cite[Cor. A.9]{QZ19} and \cite[Prop. 3.13]{IQZ}.

\begin{lemma}\label{lem:comp}
Let $\rl$ be an oriented intersection from $\tsigma$ to $\tsigma'$ and $\rl'$ an oriented intersection from $\tsigma'$ to $\tsigma''$, which are composable, see \Cref{composition of rotating angles compatability}. Let $\rl''$ be the composition of ${\rl}$ and ${\rl}'$, which is hence an oriented intersection from $\tsigma$ to $\tsigma''$. Then the following holds in $\mcC_{dg}(\Lambda)$.
\begin{equation}\notag
	f_{{\rl}'} \circ f_{{\rl}} = f_{{\rl}''}.
\end{equation}

\begin{figure}[htpb]
\begin{tikzpicture}[scale=1.8]
\draw[thick]
(-3.5,0) to (0,0);
\draw[blue]
(-1.7,0)node[below]{$M$};
\fill[fill=gray!10]
(-3.5,0) rectangle (0,-0.3);
\draw[thick,cyan]
(-1.7,0) to (-1.5,1.8)
(-1.7,0) to (-3.4,1.4)
(-1.7,0) to (0,1.2);
\draw[cyan]
(-1.7,1.6)node{$\tsigma'$}
(-2.8,1.2)node{$\tsigma$}
(-0.4,1.1)node{$\tsigma''$};
\draw[thick,Green,->-=0.5]
(-2.2,0.4) to [out=60] (-1.65,0.5);
\draw[Green](-1.95,0.75)node{${\rl}$};
\draw[thick,Green,->-=0.65]
(-1.64,0.6) to [out=40,in=120] (-1.05,0.45);
\draw[Green](-1.2,0.8)node{${\rl}'$};
\draw[thick,Green,->-=0.65]
(-2.55,0.7) to [out=50] (-0.7,0.7);
\draw[Green](-1.4,1.25)node{${\rl}''$};
\draw(-1.7,0)\bpt;
 \end{tikzpicture}
  \caption{Composition of oriented intersections.}
  \label{composition of rotating angles compatability}
 \end{figure}
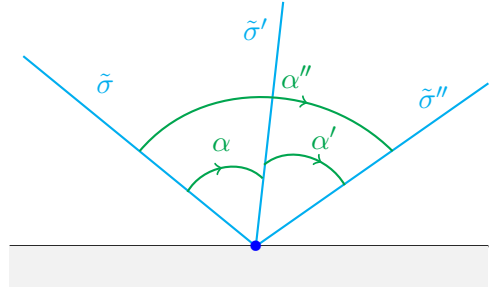
\end{lemma}


\subsection{Smoothing arcs}\label{subsec: smoothing}

 \begin{definition}
 Let $\rl$ be an oriented intersection between $\tsigma$ and $\tsigma'$ at $M \in \mfM$.
 The \emph{smoothing arc} $\tsigma \wedge_{\rl} \tsigma'$ of $\tsigma$ and $\tsigma'$ with respect to $\rl$ is obtained by gluing $\tsigma$ and $\tsigma'$ at $M$, replacing its arc segment from $\rl(0)$ to $\rl(1)$ (through $M$) with the arc $\rl$ and smoothing the arc, see \Cref{fig:smoothing}. The grading of $\tsigma \wedge_{\rl} \tsigma'$ inherits that of $\tsigma$.

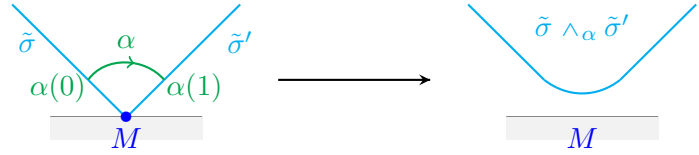
\begin{figure}[htpb]
	\begin{tikzpicture}[scale=1]
\draw[thick]
(-3,0) to (-1,0);
\draw[cyan,thick]
(-2.5,0.5) to (-3.5,1.5);
\draw[cyan]
(-3.3,1)node{$\tsigma$};
\draw[cyan]
(-0.5,1)node{$\tsigma'$};
\draw[cyan,thick]
(-1.5,0.5) to (-.5,1.5);
\draw[cyan,thick]
(-2.5,0.5) to (-2,0) to (-1.5,0.5);
\draw[Green]
(-2.9,0.4)node{$\alpha(0)$}
(-1.1,0.4)node{$\alpha(1)$};
\draw[thick,Green,->-=.6]
(-2.5,0.5) to [out=55, in=130]
(-1.5,0.5);
\draw[Green](-2,1)node{$\rl$};
\draw[thick, ->-=1,>=stealth]
(0,0.5) to (2,0.5);
\fill[fill=gray!10]
(-3,0) rectangle (-1,-0.3);
\draw(-2,0)\bpt;
\draw[blue](-2,0)node[below]{$M$};
\begin{scope}[shift={(6,0)}]
\draw[thick]
(-3,0) to (-1,0);
\draw[cyan,thick]
(-2.5,0.5) to (-3.5,1.5);
\draw[cyan,thick]
(-1.5,0.5) to (-.5,1.5);
\draw[thick,cyan]
(-2.5,0.5) to [out=-38, in=-142]
(-1.5,0.5);
\draw[cyan](-2,1.2)node{$\tsigma\wedge_{\rl}\tsigma'$};
\fill[fill=gray!10]
(-3,0) rectangle (-1,-0.3);
\draw[blue](-2,0)node[below]{$M$};
\end{scope}
\end{tikzpicture}
	\caption{Smoothing arc.}
	\label{fig:smoothing}
\end{figure}
\end{definition}

The grading  $\tsigma'\wedge_{\alpha}\tsigma$ is not the same as $\tsigma \wedge_{\rl} \tsigma'$, unless the index of $\alpha$ equals one.

\begin{remark}
For a given open (resp. closed) marked point $M$, there can be multiple oriented intersections between $\tsigma$ and $\tsigma'$. Different oriented intersections will result in different smoothing arcs. Therefore, it is necessary to specify the oriented intersections when smoothing two arcs, see \Cref{fig:smoothing1} and \Cref{fig:smoothing2}.
\end{remark}

\begin{figure}[htpb]
	\begin{tikzpicture}[scale=1.2]
\filldraw
[fill=gray!10,ultra thick](-1.3,1.3) circle (0.4);
\draw[thick]
(-3,0) to (-.3,0);
\draw[cyan,thick]
(-2,0) to [out=150,in=-90]
(-3,2);
\draw[cyan]
(-3,1)node{$\tsigma$};
\draw[cyan,thick]
(-2,0) to [out=120,in=170]
(-1,2)
(-2,0) to [out=0,in=-10]
(-1,2);
\draw[cyan](-0.5,2)node{$\tsigma'$};
\draw[thick,Green,->-=.6]
(-2.4,0.3) to [out=55, in=130]
(-1.3,0.35);
\draw[Green](-1.8,0.8)node{$\rl$};
\draw[thick, ->-=1,>=stealth]
(.5,1) to (2,1);
\fill[fill=gray!10]
(-3,0) rectangle (-0.3,-0.3);
\draw
(-2,0)\bpt;
\draw[blue](-2,0)node[below]{$M$};
\begin{scope}[shift={(6,0)}]
\filldraw
[fill=gray!10,ultra thick](-1.3,1.3) circle (0.4);
\draw[thick]
(-3,0) to (-.3,0);
\draw[cyan,thick]
(-3,2) to [out=-90,in=180]
(-2,.3);
\draw[cyan,thick]
(-2,.3) to [out=0,in=-90]
(-.7,1.2) to [out=90,in=10]
(-1.4,1.9) to [out=-170,in=70]
(-2,1.2) to [out=-110,in=120]
(-2,0);
\draw[cyan](-1.6,2.1)node{$\tsigma\wedge_{\rl}\tsigma'$};
\fill[fill=gray!10]
(-3,0) rectangle (-0.3,-0.3);
\draw
(-2,0)\bpt;
\draw[blue](-2,0)node[below]{$M$};
\end{scope}
	\end{tikzpicture}
	\caption{Smoothing with respect to $\rl$.}
	\label{fig:smoothing1}
\end{figure}
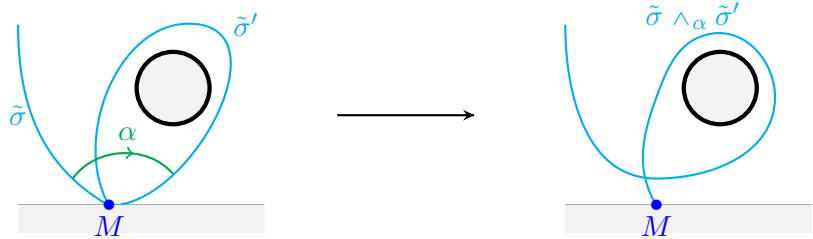

\begin{figure}[htpb]
	\begin{tikzpicture}[scale=1.2]
\filldraw
[fill=gray!10,ultra thick](-1.3,1.3) circle (0.4);
\draw[thick]
(-3,0) to (-.3,0);
\draw[cyan,thick]
(-2,0) to [out=150,in=-90]
(-3,2);
\draw[thick,Green,->-=.6]
(-2.65,0.6) to [out=55, in=145]
(-2.15,0.65);
\draw[Green](-2.4,1)node{$\rl'$};
\draw[cyan,thick]
(-2,0) to [out=120,in=170]
(-1,2)
(-2,0) to [out=0,in=-10]
(-1,2);
\draw[thick, ->-=1,>=stealth]
(.5,1) to (2,1);
\draw[cyan]
(-3,1)node{$\tsigma$};
\draw[cyan](-0.5,2)node{$\tsigma'$};
\fill[fill=gray!10]
(-3,0) rectangle (-0.3,-0.3);
\draw
(-2,0)\bpt;
\draw[blue]
(-2,0)node[below]{$M$};
\begin{scope}[shift={(6,0)}]
\filldraw
[fill=gray!10,ultra thick](-1.3,1.3) circle (0.4);
\draw[thick]
(-3,0) to (-.3,0);
\draw[cyan,thick]
(-3,2) to [out=10,in=110]
(-.6,1.5) to [out=-70,in=20]
(-2,0);
\draw[cyan](-1.9,1.9)node{$\tsigma\wedge_{\rl'}\tsigma'$};
\fill[fill=gray!10]
(-3,0) rectangle (-0.3,-0.3);
\draw
(-2,0)\bpt;
\draw[blue]
(-2,0)node[below]{$M$};
\end{scope}
	\end{tikzpicture}
	\caption{Smoothing with respect to $\rl'$.}
	\label{fig:smoothing2}
\end{figure}
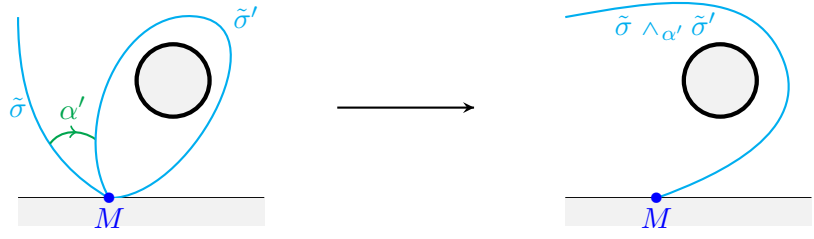

\begin{definition}\label{def:thread}
A \emph{thread} of graded open arcs is defined as a sequence of graded open arcs $\tsigma_1,\cdots,\tsigma_r$, together with oriented intersections $\alpha_i$ between $\tsigma_i$ and $\tsigma_{i+1}$, $1 \leq i \leq r-1$, written as
 \begin{gather}\label{eq:thread}
	\xymatrix{ \tsigma_{1} \ar@{-}[r]^{{\rl}_{1}} & \tsigma_{2} \ar@{-}[r]^{{\rl}_{2}} & \cdots \ar@{-}[r]^{{\rl}_{{r-1}}} & \tsigma_{r}
	},
\end{gather}
such that $\alpha_i$ and $\alpha_{i-1}$ 
do not share a common end segment of $\tsigma_i$, cf. \Cref{fig:thread}, where the right picture satisfies the condition while the left picture does not.
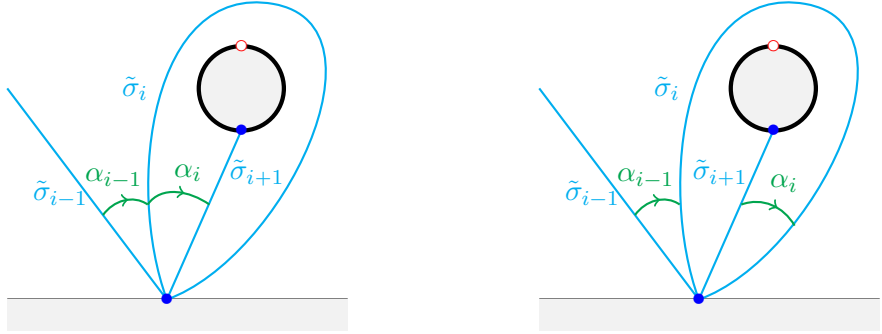
\begin{figure}[htpb]
	\begin{tikzpicture}[scale=1.4]
\filldraw
[fill=gray!10,ultra thick](-1.3,2) circle (0.4);
\draw[thick]
(-3.5,0) to (-.3,0);
\draw[cyan,thick]
(-1.3,1.6) to (-2,0)
(-2,0) to (-3.5,2)
(-2,0) to [out=110,in=170]
(-1,2.8) to [out=-10,in=20]
(-2,0);
\draw[cyan]
(-3,1)node{$\tsigma_{i-1}$}
(-1.15,1.2)node{$\tsigma_{i+1}$}
(-2.3,2)node{$\tsigma_i$};
\draw[thick,Green,->-=.6]
(-2.6,0.8) to [out=55, in=145]
(-2.18,0.9);
\draw[Green]
(-2.5,1.15)node{${\rl}_{i-1}$};
\fill[fill=gray!10]
(-3.5,0) rectangle (-0.3,-0.3);
\draw[thick,Green,->-=.6]
(-2.18,0.9) to [out=55, in=145]
(-1.6,0.9);
\draw[Green]
(-1.8,1.2)node{${\rl}_i$};
\draw
(-1.3,2.4)\rpt
(-1.3,1.6)\bpt
(-2,0)\bpt;
\begin{scope}[shift={(5,0)}]
	\filldraw
[fill=gray!10,ultra thick](-1.3,2) circle (0.4);
\draw[thick]
(-3.5,0) to (-.3,0);
\draw[cyan,thick]
(-1.3,1.6) to (-2,0)
(-2,0) to (-3.5,2)
(-2,0) to [out=110,in=170]
(-1,2.8) to [out=-10,in=20]
(-2,0);
\draw[cyan]
(-3,1)node{$\tsigma_{i-1}$}
(-1.8,1.2)node{$\tsigma_{i+1}$}
(-2.3,2)node{$\tsigma_i$};
\draw[thick,Green,->-=.6]
(-2.6,0.8) to [out=55, in=145]
(-2.18,0.9);
\draw[Green]
(-2.5,1.15)node{${\rl}_{i-1}$};
\draw[thick,Green,->-=.6]
(-1.6,0.9) to [out=20, in=120]
(-1.1,0.7);
\draw[Green]
(-1.2,1.1)node{${\rl}_i$};
\fill[fill=gray!10]
(-3.5,0) rectangle (-0.3,-0.3);
\draw
(-1.3,2.4)\rpt
(-1.3,1.6)\bpt
(-2,0)\bpt;
\end{scope}
	\end{tikzpicture}
\caption{The condition of threads.}
\label{fig:thread}
\end{figure} 



The \emph{smoothing arc} associated to this thread is denoted by $\tsigma_{1} \wedge_{{\rl}_{1}} \cdots \wedge_{{\rl}_{{r-1}}} \tsigma_{r}$, whose grading inherits that of $\tsigma_{1}$. Note that the condition for a thread ensures the well-definedness of its smoothing arc.
\end{definition}



\begin{example}\label{exm:thread}
    Let $\tsigma$ be a graded open arc. Using \Cref{not:arc}, each arc segment $\tsigma(i,i+1)$ corresponds to an oriented intersection ${\rl}_i$ of index 1 between $\tgamma_{k_i}[\Rho{\tsigma}{i}]$ and $\tgamma_{k_{i+1}}[\Rho{\tsigma}{i+1}]$ around $\M{\tsigma}{i}$, see \Cref{fig:seg to oint}. Thus, we get a thread of graded open arcs
\begin{gather}\label{eq:chu thread}
\xymatrix@C=3pc{
\talpha_{\K{\tsigma}{1}}[\Rho{\tsigma}{1}] \ar@{-}[r]^(0.6){{\rl}_{1}}  & \talpha_{\K{\tsigma}{2}}[\Rho{\tsigma}{2}] \ar@{-}[r]^(0.6){{\rl}_{2}}  & \cdots \ar@{-}[r]^(0.4){{\rl}_{{r_{\tsigma}-1}}} & \talpha_{\K{\tsigma}{r_{\tsigma}}}[\Rho{\tsigma}{r_{\tsigma}}]
	}.
\end{gather}

    \begin{figure}[htpb]
	\begin{tikzpicture}[scale=1.5]
\draw[thick]
(-5,0) to (1,0);
\fill[fill=gray!10]
(-5,0) rectangle (1,-0.3);
\draw[blue]
(-2,0)node[below]{$\M{\tsigma}{i}$};
\draw[dashed,red,thick]
(-.5,0.05) to (.3,1.55)
(-3.5,0.05) to (-4.3,1.55)
(-2.75,2.7) to (-1.2,2.7);
\draw[red,thick]
(.3,1.65) to (-1.16,2.7)
(-4.3,1.65) to (-2.85,2.7);
\draw[cyan,thick]
(-3.5,2.7) to [out=-30,in=-150] (-.5,2.7);
\draw[cyan]
(-2,2.5)node{$\tsigma(i,i+1)$};
\draw[blue,thick]
(-2,0) to (-4,2.2)
(-2,0) to (-0.2,2.5);
\draw[thick]
(-3.8,1.6)node{$\ver{\tsigma}{i}$}
(-0.3,1.8)node{$\ver{\tsigma}{i+1}$};
\draw[red,thick]
(-3.2,2.2)node{$\teta_{k_i}$}
(-0.8,2.2)node{$\teta_{k_{i+1}}$};
\draw[blue,thick]
(-3.4,0.8)node{$\tgamma_{\K{\tsigma}{i}}[\Rho{\tsigma}{i}]$}
(-0.75,0.8)node{$\tgamma_{\K{\tsigma}{i+1}}[\Rho{\tsigma}{i+1}]$};
\draw[Green,thick,->-=0.6]
(-2.5,0.55) to [out=30,in=150] (-1.6,0.55);
\draw[Green]
(-2.1,0.9)node{$\alpha_i$};
\draw
(-2,1.6)node{$\poly{\tsigma}{i}$};
\draw
(-2,0)\bpt
(-.5,0)\rpt
(-3.5,0)\rpt
(-4.3,1.6)\rpt
(-2.8,2.7)\rpt
(-1.2,2.7)\rpt
(.3,1.6)\rpt;
\end{tikzpicture}
\caption{Segment $\tsigma(i,i+1)$ induces an oriented intersection $\rl_i$.}
	\label{fig:seg to oint}
\end{figure}
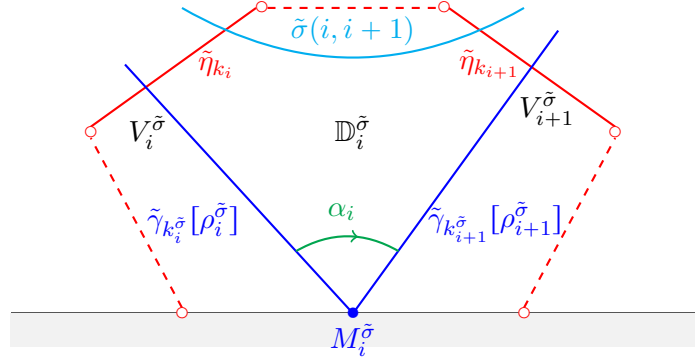
\end{example}




We introduce the following definition.
\begin{definition}
A thread of graded open arcs in \eqref{eq:thread} that satisfies the following two conditions is called a \emph{dg thread}.
\begin{enumerate}
	\item The index of each $\alpha_i$ is 1.
	\item If $f_{{\rl}_{i}}$ and $f_{{\rl}_{{i+1}}}$ have the same direction, then their composition is zero in $\mcC_{dg}(\Lambda)$.
\end{enumerate}
\end{definition}

The following lemma is useful.

\begin{lemma} \label{lem:compo}
For any $\rl\in\overrightarrow{\cap}^{1}(\tsigma,\tsigma')$, there exists a homotopy equivalence $$\chi_{\rl}: (\xymatrix{\oX(\tsigma) \ar[r]^{f_{\rl}}& \oX(\tsigma')}) \to \oX(\tsigma\wedge_{\rl}\tsigma')$$ with homotopy inverse $\psi_{\rl}$, such that the following hold. Let $\xymatrix@C=1.2pc{\tsigma_1 \ar@{-}[r]^{\rl_1} & \tsigma_2 \ar@{-}[r]^{\rl_2} & \tsigma_3}$ be a dg thread. We denote by $\rl_1'$ the induced oriented intersection between $\tsigma_1$ and $\tsigma_2\wedge_{\rl_2}\tsigma_3$ induced by $\rl_1$ see \Cref{fig:induced oint}.
\begin{enumerate}
\item If the orientation of $\rl_1$ is $\tsigma_1 \xrightarrow{{{\rl}_{{1}}}}  \tsigma_2$, 
then 
in $\mcC_{dg}(\Lambda)$, we have
\begin{equation}\label{eq:compo1}
	\chi_{{\rl}_{2}} \circ
\emb(f_{{\rl}_2},\oX(\tsigma_2)) \circ {f}_{{\rl}_{{1}}}=
f_{{\rl}'_1}.
\end{equation}
\item If the orientation of $\rl_1$ is 
$\tsigma_1 \xleftarrow{{{\rl}_{{1}}}}
\tsigma_2$, 
then 
in $\mcC_{dg}(\Lambda)$, we have
\begin{equation}\label{eq:compo2}
	{f}_{{\rl}_{{1}}} \circ \pro(f_{{\rl}_2},\oX(\tsigma_2)) \circ \psi_{{\rl}_{2}}=
 f_{{\rl}'_1}.
\end{equation}
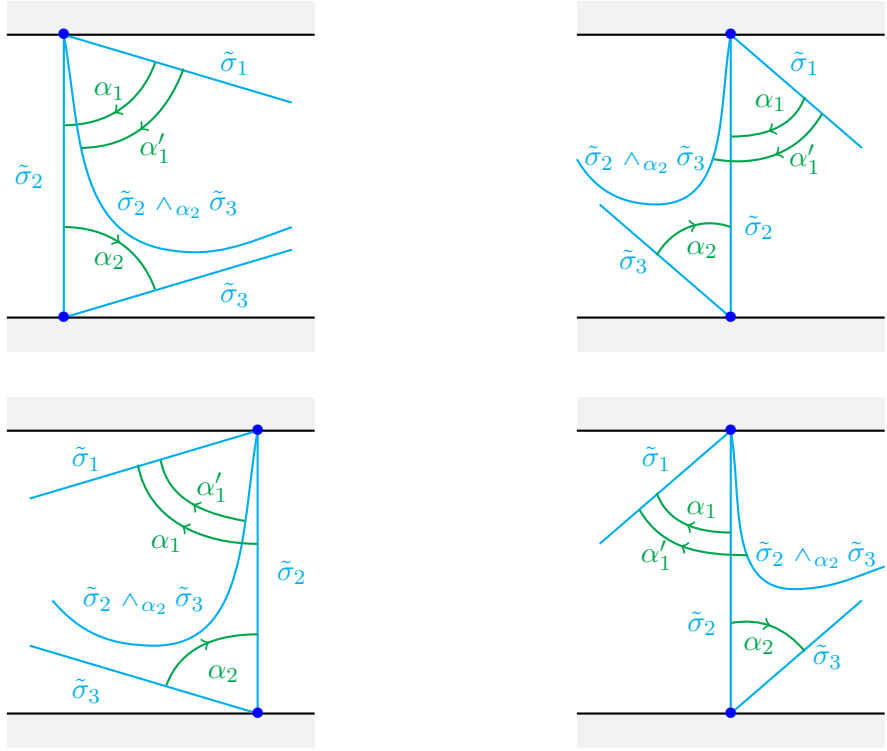
\begin{figure}[htpb]
	\begin{tikzpicture}[scale=1.5]
\fill[fill=gray!10]
(-3,0) rectangle (-0.3,-0.3);
\fill[fill=gray!10]
(-3,2.8) rectangle (-0.3,2.5);
\draw[thick]
(-3,0) to (-.3,0)
(-3,2.5) to (-.3,2.5);
\draw[cyan,thick]
(-2.5,0) to (-2.5,2.5)
(-2.5,0) to (-0.5,0.6)
(-2.5,2.5) to (-0.5,1.9);
\draw[cyan,thick]
(-1,2.25)node{$\tsigma_1$}
(-1,0.2)node{$\tsigma_3$}
(-2.8,1.25)node{$\tsigma_2$};
\draw[Green,thick,->-=0.5]
(-2.5,0.8) to [out=0,in=110] (-1.7,0.25);
\draw[Green,thick,-<-=0.5]
(-2.5,1.7) to [out=0,in=-110] (-1.7,2.25);
\draw[Green,thick]
(-2.1,0.5)node{$\rl_2$}
(-2.1,2)node{$\rl_1$};
\draw[cyan,thick]
(-0.5,0.8) to [out=-160,in=-10] (-1.6,0.6)
to [out=170,in=-80] (-2.5,2.5);
\draw[Green,thick,-<-=0.5]
(-2.35,1.5) to [out=0,in=-110] (-1.45,2.2);
\draw[Green,thick]
(-1.7,1.5)node{$\rl'_1$};
\draw[cyan,thick]
(-1.5,1)node{$\tsigma_2\wedge_{\rl_2}\tsigma_3$};
\draw
(-2.5,0)\bpt
(-2.5,2.5)\bpt;
\begin{scope}[shift={(5,0)}]
\fill[fill=gray!10]
(-3,0) rectangle (-0.3,-0.3);
\fill[fill=gray!10]
(-3,2.8) rectangle (-0.3,2.5);
\draw[thick]
(-3,0) to (-.3,0)
(-3,2.5) to (-.3,2.5);
\draw[cyan,thick]
(-1.65,0) to (-1.65,2.5)
(-1.65,0) to (-2.8,1)
(-1.65,2.5) to (-0.5,1.5);
\draw[cyan,thick]
(-1,2.25)node{$\tsigma_1$}
(-2.5,0.5)node{$\tsigma_3$}
(-1.4,0.8)node{$\tsigma_2$};
\draw[Green,thick,->-=0.6]
(-2.3,0.55) to [out=60,in=160] (-1.65,0.8);
\draw[Green,thick,-<-=0.5]
(-1.65,1.6) to [out=0,in=-110] (-1,1.95);
\draw[Green,thick]
(-1.9,0.6)node{$\rl_2$}
(-1.3,1.9)node{$\rl_1$};
\draw[cyan,thick]
(-1.65,2.5) to [out=-100,in=0]
(-2.3,1) to [out=-180,in=-60] (-3,1.4);
\draw[cyan,thick]
(-2.4,1.4)node{$\tsigma_2\wedge_{\rl_2}\tsigma_3$};
\draw[Green,thick,-<-=0.55]
(-1.8,1.4) to [out=-10,in=-120] (-0.85,1.8);
\draw[Green,thick]
(-1,1.4)node{$\rl'_1$};
\draw
(-1.65,0)\bpt
(-1.65,2.5)\bpt;
\end{scope}
\begin{scope}[shift={(0,-3.5)}]
\fill[fill=gray!10]
(-3,0) rectangle (-0.3,-0.3);
\fill[fill=gray!10]
(-3,2.8) rectangle (-0.3,2.5);
\draw[thick]
(-3,0) to (-.3,0)
(-3,2.5) to (-.3,2.5);
\draw[cyan,thick]
(-0.8,0) to (-0.8,2.5)
(-0.8,0) to (-2.8,0.6)
(-0.8,2.5) to (-2.8,1.9);
\draw[cyan,thick]
(-2.3,2.25)node{$\tsigma_1$}
(-2.3,0.2)node{$\tsigma_3$}
(-0.5,1.25)node{$\tsigma_2$};
\draw[Green,thick,-<-=0.5]
(-0.8,0.7) to [out=-180,in=70] (-1.6,0.25);
\draw[Green,thick,->-=0.5]
(-0.9,1.7) to [out=170,in=-80] (-1.65,2.24);
\draw[Green,thick]
(-1.1,0.35)node{$\rl_2$}
(-1.2,2)node{$\rl_1'$};
\draw[cyan,thick]
(-2.6,1) to [out=-50,in=180] (-1.7,0.6)
to [out=0,in=-100] (-0.8,2.5);
\draw[Green,thick,->-=0.5]
(-0.8,1.5) to [out=180,in=-80] (-1.85,2.2);
\draw[Green,thick]
(-1.6,1.5)node{$\rl_1$};
\draw[cyan,thick]
(-1.8,1)node{$\tsigma_2\wedge_{\rl_2}\tsigma_3$};
\draw
(-0.8,0)\bpt
(-0.8,2.5)\bpt;
\end{scope}
\begin{scope}[shift={(5,-3.5)}]
\fill[fill=gray!10]
(-3,0) rectangle (-0.3,-0.3);
\fill[fill=gray!10]
(-3,2.8) rectangle (-0.3,2.5);
\draw[thick]
(-3,0) to (-.3,0)
(-3,2.5) to (-.3,2.5);
\draw[cyan,thick]
(-1.65,0) to (-1.65,2.5)
(-1.65,0) to (-0.5,1)
(-1.65,2.5) to (-2.8,1.5);
\draw[cyan,thick]
(-2.3,2.25)node{$\tsigma_1$}
(-0.8,0.5)node{$\tsigma_3$}
(-1.9,0.8)node{$\tsigma_2$};
\draw[Green,thick,-<-=0.6]
(-1,0.55) to [out=130,in=10] (-1.65,0.8);
\draw[Green,thick,->-=0.5]
(-1.65,1.6) to [out=180,in=-70] (-2.3,1.95);
\draw[Green,thick]
(-1.4,0.6)node{$\rl_2$}
(-1.9,1.8)node{$\rl_1$};
\draw[cyan,thick]
(-1.65,2.5) to [out=-80,in=180]
(-1.1,1.1) to [out=0,in=-160] (-0.3,1.3);
\draw[cyan,thick]
(-0.9,1.4)node{$\tsigma_2\wedge_{\rl_2}\tsigma_3$};
\draw[Green,thick,->-=0.55]
(-1.5,1.4) to [out=180,in=-60] (-2.45,1.8);
\draw[Green,thick]
(-2.3,1.4)node{$\rl'_1$};
\draw
(-1.65,0)\bpt
(-1.65,2.5)\bpt;	
\end{scope}
\end{tikzpicture}
	\caption{Induced oriented intersections after smoothing.}
	\label{fig:induced oint}
\end{figure}
\end{enumerate}
\end{lemma}

\begin{proof}
We construct the homotopy equivalences $\chi_{{\rl}}$ and $\psi_{{\rl}}$ explicitly, using the notations in \Cref{cons:mor}. The underlying graded module of $\oX(\tsigma)\xrightarrow{f_{\rl}}\oX(\tsigma')$ is $$\bigoplus_{1\leq u\leq s}P_{i_u}[m_u]\oplus\bigoplus_{1\leq v\leq t}P_{j_v}[n_v],$$ while that of $\oX(\tsigma\wedge_{\rl}\tsigma')$ is $$\bigoplus_{k+1\leq u\leq s}P_{i_u}[m_u]\oplus\bigoplus_{k+1\leq v\leq t}P_{j_v}[n_v].$$ The non-zero components of $\chi_{\rl}$ are
\begin{itemize}
    \item $\id:P_{i_u}[m_u]\to P_{i_u}[m_u]$, $k+1\leq u\leq s$,
    \item $(-1)^k\id:P_{j_v}[n_v]\to P_{j_v}[n_v]$, $k+1\leq v\leq t$, and
    \item $(-1)^k\phi_k:P_{j_k}[n_k]\to P_{i_{k+1}}[m_{k+1}]$, when $\phi_k$ is from $P_{j_k}[n_k]$ to $P_{i_{k+1}}[m_{k+1}]$.
\end{itemize}
The non-zero components of $\psi_{\rl}$ are
\begin{itemize}
    \item $\id:P_{i_u}[m_u]\to P_{i_u}[m_u]$, $k+1\leq u\leq s$,
    \item $(-1)^k\id:P_{j_v}[n_v]\to P_{j_v}[n_v]$, $k+1\leq v\leq t$, and
    \item $\phi'_k:P_{j_{k+1}}[n_{k+1}]\to P_{i_{k}}[m_{k}]$, when $\phi'_k$ is from $P_{j_{k+1}}[n_{k+1}]$ to $P_{i_{k}}[m_{k}]$.
\end{itemize}
More precisely, depending on the forms \eqref{eq:mor1} to \eqref{eq:mor5} of $f_\alpha$ in \Cref{cons:mor}, there are the following cases, where $\chi_{{\rl}}$ and $\psi_{{\rl}}$ are drawn in green. 
\begin{enumerate}
    \item In \eqref{eq:mor1}, it is direct to check that $(\xymatrix{\oX(\tsigma) \ar[r]^{f_{\rl}}& \oX(\tsigma')})= \oX(\tsigma \wedge_{{\rl}} \tsigma')$. We take $\chi_{{\rl}}=\psi_{{\rl}}=\id$.
   \item
 In \eqref{eq:mor2}, the dg $\Lambda$-module $\oX(\tsigma \wedge_{{\rl}} \tsigma')$ is
\begin{gather} \notag
\xymatrix@C=2.5pc{P_{j_{s+1}}{[n_{s+1}]} \ar@{-}[r]^(0.6){\phi'_{s+1}} &  \cdots \ar@{-}[r]^{\phi'_{t-1}} & P_{j_t}{[n_t]}}.
\end{gather}
In this case, $\chi_{{\rl}}$ is given by the following diagram
\begin{equation} \notag
 \xymatrix@C=2pc{ P_{i_1}{[m_1]} \ar[d]^{\id} \ar@{-}[r]^(0.6){\phi_1} & \cdots  \ar[d]_{\cdots} \ar@{-}[r]^(0.4){\phi_{s-1}} & P_{i_s}{[m_s]}  \ar[d]^{(-1)^{s-1}\id}\\
    P_{j_1}{[n_1]} \ar@{-}[r]^(0.6){\phi'_1} &  \cdots  \ar@{-}[r]^(0.4){\phi'_{s-1}} & P_{j_s}{[n_s]}   & P_{j_{s+1}}{[n_{s+1}]}  \ar@{-}[r]^(0.65){\phi'_{s+1}} \ar@*{[Green]}[d]^{\G{(-1)^s\id}}\ar[l]^{\phi'_s} & \cdots \ar@[Green][d]^{\G{\cdots}} \ar@{-}[r]^(0.4){\phi'_{t-1}} & P_{j_t}{[n_t]}  \ar@[Green][d]^{\G{(-1)^s\id}} \\
     & & & P_{j_{s+1}}{[n_{s+1}]} \ar@{-}[r]_(0.65){\phi'_{s+1}} &  \cdots \ar@{-}[r]_(0.4){\phi'_{t-1}} & P_{j_t}{[n_t]}
,}
\end{equation}
and $\psi_{{\rl}}$ is given by the following diagram
\begin{gather}  \notag
 \xymatrix@C=2pc{ P_{i_1}{[m_1]} \ar[d]_{\id} \ar@{-}[r]^(0.6){\phi_1} & \cdots  \ar[d]_{\cdots} \ar@{-}[r]^(0.4){\phi_{s-1}} & P_{i_s}{[m_s]}  \ar[d]_{(-1)^{s-1}\id}\\
    P_{j_1}{[n_1]} \ar@{-}[r]_(0.6){\phi'_1} & \cdots  \ar@{-}[r]_(0.4){\phi'_{s-1}} & P_{j_s}{[n_s]}   & P_{j_{s+1}}{[n_{s+1}]}  \ar@{-}[r]^(0.65){\phi'_{s+1}} \ar@{<-}@[Green][d]^{\G{(-1)^s\id}} \ar[l]|\hole^(0.6){\phi'_s} & \cdots \ar@{<-}@[Green][d]^{\G{\cdots}} \ar@{-}[r]^(0.4){\phi'_{t-1}} & P_{j_t}{[n_t]}  \ar@{<-}@[Green][d]^{\G{(-1)^s\id}} \\
     & & & P_{j_{s+1}}{[n_{s+1}]} \ar@[Green][luu]_(0.6){\G{\phi'_s}} \ar@{-}[r]_(0.65){\phi'_{s+1}} &  \cdots \ar@{-}[r]_(0.4){\phi'_{t-1}} & P_{j_t}{[n_t]}
.}
\end{gather}
\item In \eqref{eq:mor3}, the dg $\Lambda$-module $\oX(\tsigma \wedge_{{\rl}} \tsigma')$ is
\begin{gather} \notag
\xymatrix@C=3pc{P_{i_{t+1}}{[m_{t+1}]} \ar@{-}[r]^(0.6){\phi_{t+1}} &  \cdots \ar@{-}[r]^{\phi_{s-1}} & P_{i_s}{[m_s]}}.
\end{gather}
In this case, $\chi_{{\rl}}$ is given by the following diagram
\begin{gather} \notag
 \xymatrix@C=2pc{ P_{i_1}{[m_1]} \ar[d]^{\id} \ar@{-}[r]^(0.6){\phi_1} & \cdots  \ar[d]^{\cdots} \ar@{-}[r]^(0.4){\phi_{t-1}} & P_{i_t}{[m_t]}  \ar[d]^{(-1)^{t-1}\id} \ar[r]^(0.4){\phi_t}  & P_{i_{t+1}}{[m_{t+1}]} \ar@[Green][dd]^{\Gid}  \ar@{-}[r]^(0.65){\phi_{t+1}}   & \cdots \ar@[Green][dd]^{\G{\cdots}} \ar@{-}[r]^(0.4){\phi_{s-1}} & P_{i_s}{[m_s]} \ar@[Green][dd]^{\Gid} \\
    P_{j_1}{[n_1]} \ar@{-}[r]^(0.6){\phi'_1} & \cdots  \ar@{-}[r]^(0.4){\phi'_{t-1}} & P_{j_t}{[n_t]}  \ar@[Green][rd]_(0.4){\G{(-1)^{t}\phi_t}}   \\
     & & & P_{i_{t+1}}{[m_{t+1}]} \ar@{-}[r]_(0.65){\phi_{t+1}} &  \cdots \ar@{-}[r]_(0.4){\phi_{s-1}} & P_{i_s}{[m_s]}
,}
\end{gather}
and $\psi_{{\rl}}$ is given by the following diagram
\begin{gather}  \notag
        \xymatrix@C=2pc{ P_{i_1}{[m_1]} \ar[d]^{\id} \ar@{-}[r]^(0.6){\phi_1} &  \cdots  \ar[d]^{\cdots} \ar@{-}[r]^(0.4){\phi_{t-1}} & P_{i_t}{[m_t]}  \ar[d]^{(-1)^{t-1}\id} \ar[r]^(0.4){\phi_t}  & P_{i_{t+1}}{[m_{t+1}]} \ar@[Green]@{<-}[dd]^{\Gid}  \ar@{-}[r]^(0.65){\phi_{t+1}}   & \cdots \ar@[Green]@{<-}[dd]^{\G{\cdots}} \ar@{-}[r]^(0.4){\phi_{s-1}} & P_{i_s}{[m_s]} \ar@[Green]@{<-}[dd]^{\Gid} \\
    P_{j_1}{[n_1]} \ar@{-}[r]^{\phi'_1} &  \cdots  \ar@{-}[r]^{\phi'_{t-1}} & P_{j_t}{[n_t]}     \\
     & & & P_{i_{t+1}}{[m_{t+1}]} \ar@{-}[r]_(0.65){\phi_{t+1}} &  \cdots \ar@{-}[r]_(0.4){\phi_{s-1}} & P_{i_s}{[m_s]}
.}
\end{gather}
\item In \eqref{eq:mor4}, the dg $\Lambda$-module $\oX(\tsigma \wedge_{{\rl}} \tsigma')$ is
\begin{gather}\notag
\xymatrix@C=2.5pc{ P_{i_{k+1}}{[m_{k+1}]} \ar@{-}[r]^(0.65){\phi_{k+1}} & \cdots \ar@{-}[r]^{\phi_{s-1}} & P_{i_s}{[m_s]} \\
    P_{j_{k+1}}{[n_{k+1}]}  \ar@{-}[r]^(0.65){\phi'_{k+1}} \ar[u]^{\phi_{k} \circ \phi'_{k}} & \cdots \ar@{-}[r]^{\phi'_{t-1}} & P_{j_t}{[n_t]}
    .}
\end{gather}
In this case, $\chi_{{\rl}}$ is given by the following diagram
\begin{gather}  \notag
  \xymatrix@C=2pc{ \cdots  \ar[d]^{\cdots} \ar@{-}[r]^(0.35){\phi_{k-1}} & P_{i_{k}}{[m_{k}]}  \ar[d]^{(-1)^{k-1}\id} \ar[r]^(0.5){\phi_{k}} & P_{i_{k+1}}{[m_{k+1}]} \ar@{-}[r]^(0.6){\phi_{k+1}} & \cdots \ar@{-}[r]^(0.45){\phi_{s-1}} & P_{i_s}{[m_s]}  \\
    \cdots  \ar@{-}[r]^(0.35){\phi'_{k-1}} & P_{j_{k}}{[n_{k}]} \ar@[Green][rrd]_{\G{(-1)^{k}\phi_{k}}}   & P_{j_{k+1}}{[n_{k+1}]}  \ar@{-}[r]|\hole ^(0.6){\phi'_{k+1}} \ar[l]_(0.5){\phi'_{k}} & \cdots \ar@{-}[r]|\hole^(0.6){\phi'_{t-1}} & P_{j_t}{[n_t]}    \\
      & & & P_{i_{k+1}}{[m_{k+1}]} \ar@[Green]@{<-}[luu]_(0.25){\Gid} \ar@{-}[r]^(0.6){\phi_{k+1}} & \cdots \ar@[Green]@{<-}[luu]_(0.25){\G{\cdots}}
      \ar@{-}[r]^(0.25){{\phi_{s-1}}} & P_{i_s}{[m_s]} \ar@[Green]@{<-}[luu]_(0.25){\Gid}  \\
    & & & P_{j_{k+1}}{[n_{k+1}]} \ar@[Green]@{<-}[luu]^(0.4){\G{(-1)^{k}\id}} \ar@{-}[r]_(0.6){\phi'_{k+1}} \ar[u]_{\phi_{k} \circ \phi'_{k}} & \cdots \ar@[Green]@{<-}[luu]|\hole_(0.2){\G{\cdots}}  \ar@{-}[r]_{{\phi'_{t-1}}} & P_{j_t}{[n_t]} \ar@[Green]@{<-}[luu]|\hole_(0.2){\G{(-1)^{k}\id}}
,}
\end{gather}
and $\psi_{{\rl}}$ is given by the following diagram
 \begin{gather}  \notag
   \xymatrix@C=2pc{ \cdots  \ar[d]_{\cdots} \ar@{-}[r]^(0.34){\phi_{k-1}} & P_{i_{k}}{[m_{k}]}  \ar[d]_{(-1)^{k-1}\id} \ar[r]^(0.45){\phi_{k}} & P_{i_{k+1}}{[m_{k+1}]} \ar@{-}[r]^(0.6){\phi_{k+1}} & \cdots \ar@{-}[r]^(0.4){\phi_{s-1}} & P_{i_s}{[m_s]}  \\
     \cdots  \ar@{-}[r]_(0.35){\phi'_{k-1}} & P_{j_{k}}{[n_{k}]}    & P_{j_{k+1}}{[n_{k+1}]}  \ar@{-}[r]|\hole^(0.6){\phi'_{k+1}} \ar[l]^(0.5){\phi'_{k}} & \cdots \ar@{-}[r]|\hole^(0.3){\phi'_{t-1}} & P_{j_t}{[n_t]}    \\
      & & & P_{i_{k+1}}{[m_{k+1}]} \ar@[Green][luu]_(0.7){\Gid} \ar@{-}[r]^(0.6){\phi_{k+1}} & \cdots \ar@[Green][luu]_(0.7){\G{\cdots}} \ar@{-}[r]^(0.25){\phi_{s-1}} & P_{i_s}{[n_s]} \ar@[Green][luu]_(0.7){\Gid}
     \\
    & & & P_{j_{k+1}}{[n_{k+1}]} \ar@/^/@[Green][lluuu]|(0.35){\G{\phi'_{k}}} \ar@[Green][luu]|(0.75){\G{(-1)^{k}\id}} \ar@{-}[r]_(0.6){\phi'_{k+1}} \ar[u]_{\phi_{k} \circ \phi'_{k}} & \cdots \ar@[Green][luu] |\hole|(0.75){\G{\cdots}} \ar@{-}[r]_{\phi'_{t-1}} & P_{j_t}{[n_t]} \ar@[Green][luu]|\hole |(0.75){\G{(-1)^{k}\id}}
.}
\end{gather}
\item In \eqref{eq:mor5},  the dg $\Lambda$-module $\oX(\tsigma \wedge_{{\rl}} \tsigma')$ is
\begin{gather}\notag
\xymatrix@C=2.5pc{ P_{i_{k+1}}{[m_{k+1}]} \ar@{-}[r]^(0.6){\phi_{k+1}} & \cdots \ar@{-}[r]^(0.45){\phi_{s-1}} & P_{i_s}{[m_s]} \\
    P_{j_{k+1}}{[n_{k+1}]}  \ar@{-}[r]^(0.6){\phi'_{k+1}} \ar@{<-}[u]^{\phi} & \cdots \ar@{-}[r]^(0.4){\phi'_{t-1}} & P_{j_t}{[n_t]}
    .}
\end{gather}
The morphism $\phi_k$ between $P_{i_{k+1}}[m_{k+1}]$ and $P_{i_{k}}[m_{k}]$ has the same direction of the morphism $\phi'_k$ between $P_{j_{k+1}}[n_{k+1}]$ and $P_{j_{k}}[n_{k}]$. There are two cases.
\begin{enumerate}
    \item Case that the direction of $\phi_k$ is from $P_{i_{k+1}}[m_{k+1}]$ to $P_{i_{k}}[m_{k}]$.
In this case, $\chi_{{\rl}}$ is given by the following diagram
 \begin{gather} \notag
  \xymatrix@C=2pc{  \cdots  \ar[d]_{\cdots} \ar@{-}[r]^(0.35){\phi_{k-1}} & P_{i_{k}}{[m_{k}]}  \ar[d]_{(-1)^{k-1}\id}  \ar@{<-}[r]^(0.5){\phi_{k}} & P_{i_{k+1}}{[m_{k+1}]} \ar[d]_{(-1)^{k}\phi} \ar@{-}[r]^(0.6){\phi_{k+1}} & \cdots \ar@{-}[r]^(0.4){\phi_{s-1}} & P_{i_s}{[m_s]}  \\
    \cdots  \ar@{-}[r]_(0.5){\phi'_{k-1}} & P_{j_{k}}{[n_{k}]}  & P_{j_{k+1}}{[n_{k+1}]}  \ar@{-}[r]|\hole ^(0.6){\phi'_{k+1}} \ar[l]^(0.5){\phi'_{k}} & \cdots \ar@{-}[r]|\hole^(0.6){\phi'_{t-1}} & P_{j_t}{[n_t]}    \\
      & & & P_{i_{k+1}}{[m_{k+1}]} \ar@[Green]@{<-}[luu]_(0.25){\Gid} \ar@{-}[r]^(0.6){\phi_{k+1}} & \cdots \ar@[Green]@{<-}[luu]_(0.25){\G{\cdots}}
      \ar@{-}[r]^(0.25){\phi_{s-1}} & P_{i_s}{[m_s]} \ar@[Green]@{<-}[luu]_(0.25){\Gid}  \\
    & & & P_{j_{k+1}}{[n_{k+1}]} \ar@[Green]@{<-}[luu]^(0.4){\G{(-1)^{k}\id}} \ar@{-}[r]_(0.6){\phi'_{k+1}} \ar@{<-}[u]_{\phi} & \cdots \ar@[Green]@{<-}[luu]|\hole_(0.2){\G{\cdots}}
    \ar@{-}[r]_(0.45){\phi'_{t-1}} & P_{j_t}{[n_t]} \ar@[Green]@{<-}[luu]|\hole _(0.2){\G{(-1)^{k}\id}}
,}
\end{gather}
and $\psi_{{\rl}}$ is given by the following diagram
 \begin{gather} \notag
  \xymatrix@C=2pc{ \cdots  \ar[d]_{\cdots} \ar@{-}[r]^(0.4){\phi_{k-1}} & P_{i_{k}}{[m_{k}]}  \ar[d]_{(-1)^{k-1}\id}  \ar@{<-}[r]^(0.55){\phi_{k}} & P_{i_{k+1}}{[m_{k+1}]} \ar[d]_{(-1)^{k}\phi} \ar@{-}[r]^(0.55){\phi_{k+1}} & \cdots \ar@{-}[r]^(0.4){\phi_{s-1}} & P_{i_s}{[m_s]}  \\
     \cdots  \ar@{-}[r]_(0.5){\phi'_{k-1}} & P_{j_{k}}{[n_{k}]}  & P_{j_{k+1}}{[n_{k+1}]}  \ar@{-}[r]|\hole ^(0.65){\phi'_{k+1}} \ar[l]^(0.5){\phi'_{k}} & \cdots \ar@{-}[r]|\hole^(0.6){\phi'_{t-1}} & P_{j_t}{[n_t]}    \\
      & & & P_{i_{k+1}}{[m_{k+1}]} \ar@[Green][luu]|(0.7){\Gid} \ar@{-}[r]^(0.6){\phi_{k+1}} & \cdots \ar@[Green][luu]|(0.7){\G{\cdots}} \ar@{-}[r]^(0.25){{\phi_{s-1}}} & P_{i_s}{[m_s]} \ar@[Green][luu]|(0.7){\Gid}  \\
    & & & P_{j_{k+1}}{[n_{k+1}]} \ar@/^/@[Green][lluuu]|(0.3){\G{\phi'_{k}}} \ar@[Green][luu]|(0.7){\G{(-1)^{k}\id}} \ar@{-}[r]_(0.65){\phi'_{k+1}} \ar@{<-}[u]_{\phi} & \cdots \ar@[Green][luu]|\hole|(0.7){\G{\cdots}}
    \ar@{-}[r]_{\phi'_{t-1}} & P_{j_t}{[n_t]} \ar@[Green][luu]|\hole|(0.7){\G{(-1)^{k}\id}}
.}
\end{gather}
\item Case that the direction of $\phi_k$ is from $P_{i_{k}}[m_{k}]$ to  $P_{i_{k+1}}[m_{k+1}]$.
In this case, $\chi_{{\rl}}$ is given by the following diagram
\begin{gather}  \notag
\xymatrix@C=2.5pc{\cdots  \ar[d]_{\cdots} \ar@{-}[r]^(0.35){\phi_{k-1}} & P_{i_{k}}{[m_{k}]}  \ar[d]_{(-1)^{k-1}\id}  \ar[r]^{\phi_{k}} & P_{i_{k+1}}{[m_{k+1}]} \ar[d]_{(-1)^{k}\phi} \ar@{-}[r]^(0.6){\phi_{k+1}} & \cdots \ar@{-}[r]^{\phi_{s-1}} & P_{i_s}{[m_s]}  \\
     \cdots  \ar@{-}[r]_(0.4){\phi'_{k-1}} & P_{j_{k}}{[n_{k}]} \ar@[Green][rrd]_{\G{(-1)^{k}\phi_{k}}} & P_{j_{k+1}}{[n_{k+1}]}  \ar@{-}[r]|\hole ^(0.6){\phi'_{k+1}} \ar@{<-}[l]_(0.6){\phi'_{k}} & \cdots \ar@{-}[r]|\hole^(0.6){\phi'_{t-1}} & P_{j_t}{[n_t]}    \\
 & & & P_{i_{k+1}}{[m_{k+1}]} \ar@[Green]@{<-}[luu]_(0.2){\Gid} \ar@{-}[r]^(0.6){\phi'_{k+1}} & \cdots \ar@[Green]@{<-}[luu]_(0.2){\G{\cdots}} \ar@{-}[r]^(0.6){\phi'_{s-1}} & P_{i_s}{[m_s]} \ar@[Green]@{<-}[luu]_(0.2){\Gid}  \\
 & & & P_{j_{k+1}}{[n_{k+1}]} \ar@[Green]@{<-}[luu]^(0.4){\G{(-1)^{k}\id}} \ar@{-}[r]_{\phi'_{k+1}} \ar@{<-}[u]_{\phi} & \cdots \ar@[Green]@{<-}[luu] |\hole _(0.2){\G{\cdots}} \ar@{-}[r]_(0.6){\phi'_{t-1}} & P_{j_t}{[n_t]} \ar@[Green]@{<-}[luu]|\hole _(0.2){\G{(-1)^{k}\id}}
,}
\end{gather}
and $\psi_{{\rl}}$ is given by the following diagram
\begin{gather}  \notag
\xymatrix@C=2.5pc{  \cdots  \ar[d]_{\cdots} \ar@{-}[r]^(0.35){\phi_{k-1}} & P_{i_{k}}{[m_{k}]}  \ar[d]_{(-1)^{k-1}\id}  \ar[r]^(0.5){\phi_{k}} & P_{i_{k+1}}{[m_{k+1}]} \ar[d]_{(-1)^{k}\phi} \ar@{-}[r]^(0.6){\phi_{k+1}} & \cdots \ar@{-}[r]^(0.4){\phi_{s-1}} & P_{i_s}{[m_s]}  \\
   \cdots  \ar@{-}[r]_{\phi'_{k-1}} & P_{j_{k}}{[n_{k}]}  & P_{j_{k+1}}{[n_{k+1}]}  \ar@{-}[r]|\hole ^(0.6){\phi'_{k+1}} \ar@{<-}[l]^{\phi'_k} & \cdots \ar@{-}[r]|\hole^(0.6){\phi'_{t-1}} & P_{j_t}[n_t]    \\
      &&& P_{i_{k+1}}{[m_{k+1}]} \ar@[Green][luu]_(0.2){\Gid} \ar@{-}[r]^(0.6){\phi_{k+1}} & \cdots \ar@[Green][luu]_(0.2){\G{\cdots}} \ar@{-}[r]^(0.25){\phi_{s-1}} & P_{i_s}{[m_s]} \ar@[Green][luu]_(0.2){\Gid}  \\
    & & & P_{j_{k+1}}{[n_{k+1}]} \ar@[Green][luu]^(0.35){\G{(-1)^{k}\id}} \ar@{-}[r]_{\phi'_{k+1}} \ar@{<-}[u]_{\phi} & \cdots \ar@[Green][luu] |\hole_(0.2){\G{\cdots}} \ar@{-}[r]_(0.6){\phi'_{t-1}} & P_{j_t}{[n_t]} \ar@[Green][luu]|\hole_(0.2){\G{(-1)^{k}\id}}
.}
\end{gather}
\end{enumerate}
\end{enumerate}
It can be verified that they are homotopy equivalences and homotopy inverses of each other.

By combining the constructions of $\chi_{{\rl}}$, $\psi_{{\rl}}$ and the morphisms between $\oX(\tsigma)$ and $\oX(\tsigma')$ in \Cref{cons:mor}, \eqref{eq:compo1} and \eqref{eq:compo2} are obtained.
\end{proof}

Each dg thread \eqref{eq:thread} gives rise to a dg $\Lambda$-module
\begin{gather} \label{eq:dgm from thread}
  \xymatrix{ \oX(\tsigma_1) \ar@{-}[r]^{f_{{\rl}_{1}}} & \oX(\tsigma_2) \ar@{-}[r]^{f_{{\rl}_{2}}} & \cdots  \ar@{-}[r]^(0.4){f_{{\rl}_{{n-1}}}} & {\oX(\tsigma_n)}
    }.
\end{gather}

Our main theorem in this subsection is the following.
\begin{theorem}\label{thm:smoothing}
For a given dg thread of graded open arcs as described in \eqref{eq:thread}, there is an isomorphism in $\per \Lambda$
\begin{gather}\label{eq:isosmoothing}
\xymatrix{ (\oX(\tsigma_1) \ar@{-}[r]^{f_{{\rl}_{1}}} &\oX(\tsigma_2) \ar@{-}[r]^{f_{{\rl}_{2}}} & \cdots \ar@{-}[r]^(0.4){f_{{\rl}_{{n-1}}}} & {\oX(\tsigma_n)})}\cong \oX(\tsigma_{1} \wedge_{{\rl}_{1}}\tsigma_{2} \wedge_{{\rl}_{2}} \cdots \wedge_{{\rl}_{{n-1}}} \tsigma_{n}).
\end{gather}
\end{theorem}

Taking $n=2$ in \Cref{thm:smoothing}, similar results can be found in \cite{QZ19} and \cite{IQZ} in different settings.


\begin{proof}
We use induction on $n$. The case $n=1$ is trivial. The case $n=2$ follows directly from \Cref{lem:compo}. Now assume that the required isomorphism exists for the case $n-1$, and we consider the case $n$ for $n \geq 3$.
Denote by $\rl_{n-2}'$ the oriented intersection between $\tsigma_{n-2}$ and $\tsigma_{n-1}\wedge_{\rl_{n-1}}\tsigma_{n}$, induced by $\rl_{n-2}$. Then we have a thread
\begin{equation}\label{eq:indthread}
   \xymatrix{ \tsigma_1 \ar@{-}[r]^{{\rl}_{1}} & \cdots \ar@{-}[r]^{{{\rl}_{{n-3}}}} & \tsigma_{n-2}  \ar@{-}[r]^(0.35){{{\rl}'_{{n-2}}}} & {\tsigma_{n-1} \wedge_{{\rl}_{{n-1}}}\tsigma_n}
    .}
\end{equation}
Depending on the orientation of $\rl_{n-2}$, there are the following two cases.
\begin{enumerate}
    \item The case $\tsigma_{n-2}\xrightarrow{\rl_{n-2}}\tsigma_{n-3}$. In this case, $\rl'_{n-2}$ points to the right. then by \Cref{lem:compo}, we have
    \begin{equation}\label{eq:lin1}
        f_{{\rl}'_{n-2}}=\chi_{{\rl}_{n-1}} \circ
\emb(f_{{\rl}_{n-1}},\oX(\tsigma_{n-1})) \circ {f}_{{\rl}_{n-2}}.
    \end{equation}
    Hence if $\rl_{n-3}$ also points to the right, we have $f_{{\rl}'_{n-2}}\circ f_{{\rl}_{n-3}}=0$. This implies that \eqref{eq:indthread} is a dg thread. Consider the following diagram
    \begin{gather} \label{eq:pfsmooth1}
\xymatrix@C=4pc{\oX(\tsigma_1) \ar[d]^{\id} \ar@{-}^{f_{{\rl}_{1}}}[r] & \cdots \ar[d]^{\cdots}\ar@{-}^{f_{{\rl}_{{n-3}}}}[r] &  \oX(\tsigma_{n-2}) \ar[d]_{\id} \ar[r]^(0.4){\emb'} & \Cone( f_{{\rl}_{{{n-1}}}}) \ar[d]^{\chi_{{\rl}_{{n-1}}}} \\
\oX(\tsigma_1) \ar@{-}^{f_{{\rl}_{1}}}[r] & \cdots \ar@{-}^{f_{{\rl}_{{n-3}}}}[r] &  \oX(\tsigma_{n-2}) \ar[r]^(0.4){f_{{\rl}'_{{{n-2}}}}} &  \oX(\tsigma_{n-1} \wedge_{{\rl}_{{n-1}}} \tsigma_n)
,}
\end{gather}
where $\emb'=\emb(f_{{\rl}_{{n-1}}},\oX(\tsigma_{n-1})) \circ {f}_{\rl_{{n-2}}}$. By \eqref{eq:lin1}, this diagram commutes and hence can be viewed as a dg $\Lambda$-module morphism from the first row to the second row. Moreover, since $\chi_{\rl_{n-1}}$ is a homotopy equivalence by \Cref{lem:compo}, this morphism is a quasi-isomorphism. Note that the first row of \eqref{eq:pfsmooth1} is the same as the left-hand side of \eqref{eq:isosmoothing}, while the second row, by applying the induction assumption to \eqref{eq:indthread}, is isomorphic to the right-hand side of \eqref{eq:isosmoothing}. Therefore, we are done.
\item The case $\tsigma_{n-2}\xleftarrow{\rl_{n-2}}\tsigma_{n-3}$. The proof is similar, after replacing \eqref{eq:lin1} with
\begin{equation}\label{eq:lin2}
        f_{{\rl}'_{n-2}}={f}_{{\rl}_{n-2}} \circ
\pro(f_{{\rl}_{n-1}},\oX(\tsigma_{n-1})) \circ \psi_{{\rl}_{n-1}},
    \end{equation}
    and replacing the diagram \eqref{eq:pfsmooth1} by
    \begin{gather} \label{eq:pfsmoothing2}
\xymatrix@C=4pc{\oX(\tsigma_1) \ar@{<-}[d]^{\pm\id} \ar@{-}^{f_{{\rl}_{1}}}[r] & \cdots \ar@{<-}[d]^{\cdots}\ar@{-}^{f_{{\rl}_{{n-3}}}}[r] &  \oX(\tsigma_{n-2}) \ar@{<-}[d]_{\pm\id} \ar@{<-}[r]^(0.4){\pro'} & \Cone( f_{{\rl}_{{{n-1}}}}) \ar@{<-}[d]^{\psi_{{\rl}_{{n-1}}}} \\
\oX(\tsigma_1) \ar@{-}^{f_{{\rl}_{1}}}[r] & \cdots \ar@{-}^{f_{{\rl}_{{n-3}}}}[r] &  \oX(\tsigma_{n-2}) \ar@{<-}[r]^(0.4){f_{{\rl}'_{{{n-2}}}}} &  \oX(\tsigma_{n-1} \wedge_{{\rl}_{{n-1}}} \tsigma_n)
,}
\end{gather}
where $\pro'={f}_{{\rl}_{{n-2}}} \circ \pro(f_{{\rl}_{{n-1}}},\oX(\tsigma_{n-1}))$.
\end{enumerate}

\end{proof}
\subsection{Independence}\label{independence}\

Let $\mfA'=\{\talpha_1',\cdots,\talpha_n'\}$ be another full formal open arc system on $\mfS^{\lambda}$. Since both $\per\Lambda_{\mfA}$ and $\per\Lambda_{\mfA'}$ are the topological Fukaya category of $\mfS^{\lambda}$, see \cite{HKK}, there exists a triangle equivalence $\per\Lambda_{\mfA} \xrightarrow{\simeq}\per\Lambda_{\mfA'}$. This subsection is devoted to, under a certain condition, constructing explicitly such an equivalence which is compatible with \Cref{cons:obj}.


\begin{definition}\label{def:str f}
We call $\mfU'$ \emph{strongly formal with respect to $\mfU$} provided that for any two oriented intersections $\rl,\rl'$, which are two adjacent inner angles of the same $\mfU'$-polygon, see \Cref{fig:strong}, we have $f_{{\rl'}} \circ f_{{\rl}}=0$.

\begin{figure}[htpb]
\begin{tikzpicture}
\draw[thick](-4,0) to (-1,0)
(0,0) to (3,0);
\draw[cyan,thick]
(-2.5,0) to [out=50,in=-180] (-0.5,0.7) to [out=0,in=130]
(1.5,0)
(-2.5,0) to (-3.5,2.5)
(1.5,0) to (2.5,2.5);
\draw[cyan]
(-3,2)node{$\tgamma'_i$}
(-0.5,1)node{$\tgamma'_j$}
(2,2)node{$\tgamma'_k$};
\draw[thick]
(-0.5,2.5)node{$f_{{\rl'}} \circ f_{\rl}=0$.};
\draw[thick,Green,->-=0.5]
(-2.8,0.7) to [out=30] (-1.8,0.5);
\draw[Green](-2.2,1.1)node{${\rl}$};
\draw[thick,Green,-<-=0.5]
(1.8,0.7) to [out=-200,in=50] (0.8,0.5);
\draw[Green](1.35,1.1)node{${\rl'}$};
\fill[fill=gray!10]
(-4,0) rectangle (-1,-0.3)
(0,0) rectangle (3,-0.3);
\draw(-2.5,0)\bpt
(1.5,0)\bpt;
\end{tikzpicture}
\caption{Condition of strongly formal arc systems.}
\label{fig:strong}
\end{figure}
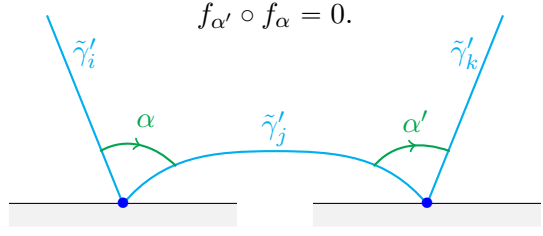
\end{definition}

Let $\oX(\mfA')=\oX(\talpha'_1)\oplus \cdots \oplus \oX(\talpha'_n)$ and denote by $\thick_{\mcD(\Lambda_{\mfA})}(\oX(\mfA'))$ the smallest thick subcategory of $\mcD(\Lambda_{\mfA})$ containing $\oX(\mfA')$.

\begin{remark}\label{rmk:thread1}
    Suppose $\mfA'$ is strongly formal with respect to $\mfU$. For any graded open arc $\tsigma$, as in \Cref{exm:thread} (but replacing $\mfU$ with $\mfU'$), there is a dg thread of graded open arcs in $\mfU'$ 
    \begin{equation}\label{eq:thread1}
    \xymatrix{
    \tgamma_{i_1}'[m_1]
    \ar@{-}[r]^{\rl_1} &
    \tgamma_{i_2}'[m_2]
    \ar@{-}[r]^{\rl_2} &
    \cdots
    \ar@{-}[r]^(0.35){\rl_{r-1}} &
    \tgamma_{i_r}'[m_r],
    }
\end{equation}
with each $\rl_i$ of index 1, and such that
$\tsigma=\talpha'_{i_1}[m_1]\wedge_{{\rl}_1}\talpha'_{i_2} [m_2]\wedge_{{\rl}_2} \cdots \wedge_{{\rl}_{r-1}} \talpha'_{i_r}[m_r]$. Since $\mfA'$ is strongly formal with respect to $\mfU$, the thread \eqref{eq:thread1} is a dg thread. Therefore, by \Cref{thm:smoothing}, we have an isomorphism
\begin{equation}\label{eq:cong1}
\oX(\tsigma)\cong\left(\xymatrix{
\oX(\tgamma_{i_1}')[m_1]
\ar@{-}[r]^{f_{\rl_1}} &
\oX(\tgamma_{i_2}')[m_2]
\ar@{-}[r]^(0.65){f_{\rl_2}} &
\cdots
\ar@{-}[r]^(0.35){f_{\rl_{r-1}}} &
\oX(\tgamma_{i_r}')[m_r]
}\right).
\end{equation}
In particular, we have $\oX(\tsigma)\in\thick_{\mcD(\Lambda_{\mfA})}(\oX(\mfA'))$. This implies 
\begin{equation}\label{eq:=1}
\thick_{\mcD(\Lambda_{\mfA})}(\oX(\mfA'))=\per\Lambda_{\mfA}.
\end{equation}
\end{remark}

Denote by  $\Lambda'=\mcHom_{\Lambda}(\oX(\mfA'),\oX(\mfA'))$. The following lemma is a standard trick on $A_{\infty}$-algebra, cf. \cite{Ke01}. We reduce it to the dg setting here.

\begin{lemma}\label{lem:quasi}
Suppose $\mfA'$ is strongly formal with respect to $\mfU$. Then $\Lambda_{\mfU'}=H^*(\Lambda')$ and there is a dg algebra quasi-isomorphism
$$q_{\mfA'}:H^*(\Lambda') \rightarrow \Lambda',$$
such that $q_{\mfA'}(\bar{f}_{\rl})=f_{\rl}$ for any oriented intersection $\rl$ between arcs in $\mfA'$, where $\bar{f}_{\rl}$ denotes the image of $f_{\rl}\in\Lambda'$ in $H^\ast(\Lambda')$.
\end{lemma}

\begin{proof}
Let $f_{\mfA'}$ be the set of $f_{\rl}$, for oriented intersections $\rl$ between arcs in $\mfA'$. By \Cref{thm:QZZ}, the image of $f_{\mfA'}$ in $\per\Lambda$ is a basis of $\bigoplus_{\rho\in\mbZ}\Hom_{\per\Lambda}(\oX(\mfA'),
\oX(\mfA')[\rho])=H^*(\Lambda')$. 
Since $\mfA'$ is strongly formal with respect to $\mfA$, together with \Cref{lem:comp}, $f_{\mfA'}$ is closed under compositions of morphisms, and by \Cref{rmk:rebuilt}, one can identify $H^*(\Lambda')$ with $\Lambda_{\mfA'}$. The former allows us to linearly extend $f_{\mfA'}$ to a subalgebra of $\Lambda'$, which is isomorphic to $H^*(\Lambda')$ as graded algebras. Thus the map sending $\bar{f}_{\rl}$ to $f_{\rl}$ gives rise to the required dg algebra quasi-isomorphism.
\end{proof}

    


Recall that for a graded open arc $\tsigma$, we construct a dg $\Lambda_{\mfU}$-module $\oX(\tsigma)$, as described in \Cref{cons:obj}. Similarly, we can construct a dg $\Lambda_{\mfU'}$-module using the same method, denoted by $\oX'(\tsigma)$.

\begin{theorem}\label{thm:independence}
Suppose that $\mfU'$ is strongly formal with respect to $\mfU$. Then there is a triangle equivalence
\begin{equation}\label{Independence derived equivalence}
\Phi_{\mfA'}
:\per \Lambda_{\mfU} \rightarrow \per \Lambda_{\mfU'},	
\end{equation}
such that for any graded open arc $\tsigma$, we have an isomorphism in $\per\Lambda_{\mfA}$
\begin{equation}\label{eq:cong}
    \Phi_{\mfA'}(\oX(\tsigma))\cong \oX'(\tsigma).
\end{equation}
\end{theorem}
\begin{proof}

By \Cref{lem:quasi}, there is a dg algebra quasi-isomorphism
$$q_{\mfA'}:\Lambda_{\mfU'}= H^*(\Lambda') \rightarrow \Lambda'.$$
Then by \Cref{thm:quasi-iso}, $q_{\mfA'}$ is also a dg $\Lambda_{\mfA'}$-module quasi-isomorphism from $\Lambda_{\mfA'}$ to $(\Lambda')_{\Lambda_{\mfA'}}=\rest{q_{\mfA'}}((\Lambda')_{\Lambda'})$, where $\rest{q_{\mfA'}}:\mathcal{C}_{dg}(\Lambda')\to\mathcal{C}_{dg}(\Lambda_{\mfA'})$ is the restriction functor along $q_{\mfA'}$. 

Denote $\mathcal{P}'_i=\mcHom_{\Lambda_{\mfU}}(\oX(\mfA'),\oX(\talpha'_i))$ and $P'_i=H^*(\mathcal{P}'_i)$, for any $1\leq i\leq n$.  Then we have 
$$\Lambda'=\bigoplus_{i=1}^n\mathcal{P}'_i,\text{ and }\Lambda_{\mfU'}=\bigoplus_{i=1}^n P'_i.$$
So $q_{\mfA'}$ induces dg $\Lambda_{\mfA'}$-module quasi-isomorphisms from $P_i'$ to $(\mathcal{P}_i')_{\Lambda_{\mfA'}}=\rest{q_{\mfA'}}((\mathcal{P}_i')_{\Lambda'})$, $1\leq i\leq n$, which are also denoted by $q_{\mfA'}$. Note that $P'_i$ is the direct summand of the dg $\Lambda_{\mfA'}$-module $\Lambda_{\mfA'}$, corresponding to the idempotent associated to $\tgamma'_i$. For any oriented intersection $\rl$ from $\tgamma'_i$ to $\tgamma'_j$, the dg $\Lambda$-module morphism $f_{\rl}:\oX(\tgamma'_i)\to \oX(\tgamma'_j)$ induces a dg $\Lambda'$-module morphism:
$f_{\rl}\circ -:\mathcal{P}'_i\to\mathcal{P}'_j$, which after the action of $\rest{q_{\mfA'}}$, becomes a dg $\Lambda_{\mfA'}$-module morphism $f_{\rl}\circ -:(\mathcal{P}'_i)_{\Lambda_{\mfA'}}\to(\mathcal{P}'_j)_{\Lambda_{\mfA'}}$. Since by Lemma~\ref{lem:quasi}, $q_{\mfA'}(\bar{f}_{\rl})=f_{\rl}$, we have the following commutative diagram of dg $\Lambda_{\mfA'}$-module morphisms
\begin{equation}\label{eq:diagram1}
\xymatrix{
P'_i\ar[d]_{q_{\mfA'}}\ar[r]^{\bar{f}_{\rl}\circ -}&P'_j\ar[d]^{q_{\mfA'}}\\
(\mathcal{P}'_i)_{\Lambda_{\mfA'}}\ar[r]^{f_{\rl}\circ -}&(\mathcal{P}'_j)_{\Lambda_{\mfA'}}\ \ .
}
\end{equation}
For any graded open arc $\tsigma$, we use the notations in Remark~\ref{rmk:thread1}, especially the dg thread \eqref{eq:thread1} and the isomorphism \eqref{eq:cong1}. On the one hand, taking the composition of dg functors 
\begin{equation}\label{eq:Phi}
    \Phi_{\mfA'}=\rest{q_{\mfA'}}\circ\mcHom_{\Lambda_{\mfU}}(\oX(\mfA'),-):\per\Lambda_{\mfA} \rightarrow \per\Lambda_{\mfA'}
   ,
\end{equation}
and applying it to \eqref{eq:cong1}, we have an isomorphism of dg $\Lambda_{\mfU'}$-module
\begin{gather}\notag
\Phi_{\mfA'}(X(\tsigma))\cong\left(\xymatrix{
(\mathcal{P}_{i_1}')_{\Lambda_{\mfA'}}[m_1]
\ar@{-}[r]^{f_{\alpha_1}\circ-} &
(\mathcal{P}_{i_2}')_{\Lambda_{\mfA'}}[m_2]
\ar@{-}[r]^(0.7){f_{\alpha_2}\circ-} &
\cdots
\ar@{-}[r]^(0.3){f_{\alpha_{r-1}}\circ-} &
(\mathcal{P}_{i_r}')_{\Lambda_{\mfA'}}[m_r]
}\right).
\end{gather}
On the other hand, by applying $\oX'(-)$ to \eqref{eq:thread1}, we have a dg $\Lambda_{\mfA'}$-module
$$\oX'(\tsigma)=
\xymatrix{
P'_{i_1}[m_1]
\ar@{-}[r]^{\bar{f}_{\alpha_1}\circ-} &
P'_{i_2}[m_2]
\ar@{-}[r]^{\bar{f}_{\alpha_2}\circ-} &
\cdots
\ar@{-}[r]^{\bar{f}_{\alpha_{r-1}}\circ -\quad} &
P'_{i_r}[m_r]
	}.$$
Hence, by \eqref{eq:diagram1}, there is a commutative diagram
$$\xymatrix@C=2.5pc{
	P'_{i_1}[m_1] \ar[d]^{q_{\mfA'}}
\ar@{-}[r]^{\bar{f}_{\alpha_1}\circ-} &
P'_{i_2}[m_2] \ar[d]^{q_{\mfA'}}
\ar@{-}[r]^(0.55){\bar{f}_{\alpha_2}\circ-} &
\cdots \ar[d]^{\cdots}
\ar@{-}[r]^(0.4){\bar{f}_{\alpha_{r-1}}\circ-} &
P'_{i_r}[m_r] \ar[d]^{q_{\mfA'}} \\
	(\mathcal{P}_{i_1}')_{\Lambda_{\mfA'}}[m_1]
\ar@{-}[r]^{f_{\rl_1}\circ-} &
(\mathcal{P}_{i_2}')_{\Lambda_{\mfA'}}[m_2]
\ar@{-}[r]^(0.65){f_{\rl_2}\circ-} &
\cdots
\ar@{-}[r]^(0.3){f_{\rl_{r-1}}\circ-} &
(\mathcal{P}_{i_r}')_{\Lambda_{\mfA'}}[m_r]\ ,}$$
which is hence a dg $\Lambda_{\mfA'}$-module morphism from the first row to the second row. This is indeed a quasi-isomorphism because each of $q_{\mfA'}$ is. Thus, we have an isomorphism
$$\Phi_{\mfA'}(\oX(\sigma))\cong \oX'(\sigma)$$
in $\per\Lambda_{\mfA'}$. 
Further, by \eqref{eq:=1} and \Cref{thm:keller}, $\mcHom_{\Lambda_{\mfU}}(\oX(\mfA'),-)$ induces a triangle equivalence from $\per \Lambda_{\mfU}$ to $\per \Lambda'$. By \Cref{thm:quasi-iso}, $\rest{q_{\mfA'}}$ induces a triangle equivalence from $\per \Lambda'$ to $\per \Lambda_{\mfU'}$. So $\Phi_{\mfA'}$ is a triangle equivalence from $\per \Lambda_{\mfU}$ to $\per \Lambda_{\mfU'}$.

\end{proof}

\section{A geometric realization of Koszul duality}\label{geometric model for koszul duality}


In this section, we devote to giving a geometric realization of Koszul duality of graded gentle algebras.


We first introduce an operation on the graded arcs on $\mfS^{\lambda}$.

\begin{definition}\label{def:half rot}
    Let $\tbeta$ be a graded arc whose endpoints are in $\mfM \bigcup \mfY$. The \emph{(clockwise) half rotation} $\HR(\tbeta)$ of $\tbeta$ is obtained by moving the endpoints of $\tbeta$ to the next marked points along the boundary anti-clockwise (cf. \Cref{half flip}). The grading of $\HR(\tbeta)$ inherits that of $\tbeta$. 
\end{definition}

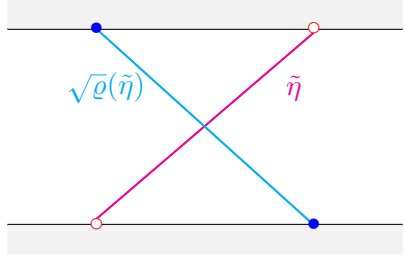
\begin{figure}[htpb]
\begin{tikzpicture}[scale=1.3]
\draw[thick]
(-2,0) to (2,0)
(-2,2) to (2,2);
\fill[fill=gray!10]
(-2,0) rectangle (2,-0.3)
(-2,2.3) rectangle (2,2);
\draw[thick,magenta]
(-1.1,0.06) to (1.1,1.96);
\draw[magenta]
(.9,1.4)node{$\tbeta$};
\draw[thick,cyan]
(1.1,0) to (-1.1,2);
\draw[cyan](-1,1.4)node{$\HR(\tbeta)$};
\draw
(-1.1,0)\rpt
(1.1,0)\bpt
(-1.1,2)\bpt
(1.1,2)\rpt;
\end{tikzpicture}
\caption{Half rotation of a graded closed arc.}
\label{half flip}	
\end{figure}

By definition, the half rotation $\HR$ gives rise to bijections from $\OC$ to $\CC$, and from $\CC$ to $\OC$, respectively. The inverse of $\HR$, denoted by $\AHR$, is called the \emph{anti-clockwise half rotation}. Note that $\HR^2=\varrho$ is the usual rotation on the sets $\OC$ and $\CC$, which is a geometric realization of the Auslander-Reiten translation in the ungraded case, cf. \cite{OPS,CS}.

Recall that we have fixed a full formal open arc system $\mfU=\{\tgamma_1,\cdots,\tgamma_n\}$ and its dual $\mfV=\{\teta_1,\cdots,\teta_n\}$. Denote $\HR(\mfV)=\{\HR(\tbeta_1),\cdots,\HR(\tbeta_n)\}$. By definition, $\HR(\mfV)$ is a full formal open arc system.

\begin{lemma}\label{lem:str f}
    The full formal open arc system $\HR(\mfV)$ is strongly formal with respect to $\mfU$.
\end{lemma}

\begin{proof}
    Denote $\tgamma'_i=\HR(\tbeta_i)$ for $1\leq i\leq n$. Assume that there are different $\tgamma'_i,\tgamma'_j,\tgamma'_k$ and oriented intersections $\rl,\rl'$ between them, which are in the situation in \Cref{def:str f}, see \Cref{fig:strong}. Due to \Cref{lem:comp}, we may assume that $\tbeta_i$, $\tbeta_j$ and $\tbeta_k$ are three successive edges of the same $\mfV$-polygon, see \Cref{fig:half rot str}. Then the relative position of $\tgamma'_i$ and $\tgamma'_j$ is of the form shown in the middle picture of \Cref{fig:mor45}, while the relative position of $\tgamma'_j$ and $\tgamma'_k$ is of the form shown in the middle picture of \Cref{fig:mor45} or the form shown in the right picture in \Cref{fig:mor123}. Hence by \Cref{cons:mor}, all of the non-zero components of $f_{\rl}$ and $f_{\rl'}$ are identities, and the codomain of each identity in $f_{\rl}$ is different from the domain of any identity in $f_{\rl'}$. So we have $f_{\rl'} \circ f_{\rl}=0$. This implies that $\HR(\mfV)$ is strongly formal with respect to $\mfU$.

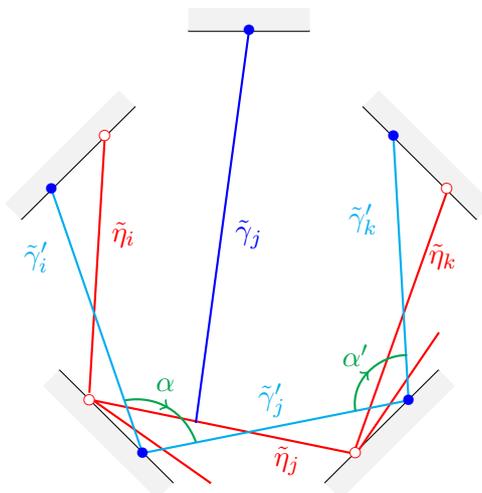
\begin{figure}[htpb]
	\begin{tikzpicture}
\draw[thick]
(-3,0) to (-4.5,1.5)
(-1,0) to (0.5,1.5)
(-3.5,5) to (-5,3.5)
(-0.5,5) to (1,3.5)
(-2.8,6) to (-1.2,6);
\fill[fill=gray!10]
(-2.8,6.3) rectangle (-1.2,6)
(-4.5,1.5)--(-4.71,1.29)--(-3.21,-0.21)--(-3,0)
(0.5,1.5)--(-1,0)--(-0.79,-0.21)--(0.71,1.29)
(-3.71,5.21)--(-5.21,3.71)--(-5,3.5)--(-3.5,5)
(-0.29,5.21)--(-0.5,5)--(1,3.5)--(1.21,3.71);
\draw[thick,red]
(-4.1,1.2) to (-3.9,4.55)
(-4.05,1.1) to (-0.65,0.4)
(-0.6,0.46) to (0.6,3.85)
(-4.05,1.1) to (-2.5,0)
(-0.55,0.45) to (0.5,2);
\draw[thick,cyan]
(-3.4,0.4) to (-4.6,3.9)
(-3.4,0.4) to (0.1,1.1)
(0.1,1.1) to (-0.1,4.6);
\draw[thick,blue]
(-2,6) to (-2.7,0.8);
\draw[red]
(-3.65,3.3)node{$\teta_i$}
(-1.5,0.3)node{$\teta_j$}
(0.55,3)node{$\teta_k$};
\draw[cyan]
(-4.8,3)node{$\tgamma'_i$}
(-1.7,1.1)node{$\tgamma'_j$}
(-0.5,3.5)node{$\tgamma'_k$};
\draw[thick,blue](-2,3.3)node{$\tgamma_j$};
\draw[Green, thick,->-=0.5]
(-3.63,1.1) to [out=10,in=110] (-2.7,0.55);
\draw[Green, thick,->-=0.5]
(-0.6,0.96) to [out=100,in=-180] (0.08,1.7);
\draw[Green,thick]
(-3.1,1.3)node{$\rl$}
(-0.6,1.7)node{$\rl'$};
\draw
(-3.4,0.4)\bpt
(0.1,1.1)\bpt
(-4.6,3.9)\bpt
(-0.1,4.6)\bpt
(-2,6)\bpt
(-4.1,1.1)\rpt
(-0.6,0.4)\rpt
(-3.9,4.6)\rpt
(0.6,3.9)\rpt;
	\end{tikzpicture}
	 \caption{Half rotation of dual arc system.}
	 \label{fig:half rot str}
\end{figure}
\end{proof}

We denote by $S_i$ the simple dg $\Lambda_{\mfA}$-module corresponding to the vertex $i$. A path $a_1a_2\cdots a_l$ in $Q_{\mfA}$ (i.e. $a_j\in (Q_{\mfA})_1$ with $t(a_j)=s(a_{j+1})$) is called a maximal relation path if each $a_ja_{j+1} \in I$ and  there is no arrow $a_{l+1} \in (Q_{\mfA})_1$ such that $s(a_l)=t(a_{l+1})$ and $a_la_{l+1} \in I$. 
Since $(Q_{\mfA},I_{\mfA})$ is a gentle pair, there is a bijection between maximal relation paths and their starting arrows.

\begin{lemma} \label{lem:S}
    There is a quasi-isomorphism $\pi_i\colon\oX(\HF(\tbeta_i))\to S_i$ of dg $\Lambda_{\mfA}$-modules.
\end{lemma}

\begin{proof}
    There is an induced intersection $Z_i$ between $\HF(\tbeta_i)$ and $\tbeta_i$, such that each half of $\HF(\tbeta_i)$ divided by $Z_i$ goes through a sequence of arc segments, which corresponds to a maximal relation path starting at $i$, see \Cref{flip of simples}. Hence by \Cref{cons:obj}, $\oX(\HF(\tbeta_i))$ is given by the following form
    \begin{gather} \label{resolution of simple has 2 arrows out}
    \xymatrix@C=3pc{
        P_{t(a_l)}{[l-(\sum^{l}_{i=1}|a_i|)]}\ar[r]^(0.7){\phi_l} & \cdots \ar[r]^(0.3){\phi_2} & P_{t(a_1)}{[1-|a_1|]} \ar[rd]^{\phi_1} & \\
        &&& P_i, \\
        P_{t(a'_m)}{[m-(\sum^{m}_{j=1}|a'_j|)]}\ar[r]^(0.7){\phi'_m} & \cdots \ar[r]^(0.3){\phi'_2} & P_{t(a'_1)}{[1-|a'_1|]} \ar[ru]^{\phi'_1}
    }
    \end{gather}
    where $a_1\cdots a_l$ and $a'_1\cdots a'_m$ are the maximal relation paths starting at $i$, $\phi_i$ is induced by $\phi_{a_i}$ and $\phi'_j$ is induced by $\phi_{a_j'}$. Note that any of the two maximal relation paths may not exist. Let $\pi_i$ be the morphism from $\oX(\HF(\tbeta_i))$ to $S_i$ whose only non-zero component is the canonical projection form $P_i$ to $S_i$. Then $\pi_i$ is a dg $\Lambda$-module morphism with $H^0(\pi_i)$ an isomorphism. Hence $\pi_i$ is a quasi-isomorphism.


\begin{figure}[htpb]
\begin{tikzpicture}[scale=1]

\draw[thick]
(-4,0) to (0,0)
(-4,3) to (0,3);
\draw[thick,red]
(-2,0.06) to (-2,2.96)
(-2,0.06) to (0.5,0.6)
(-2,0.06) to (-00.5,2.5)
(-2,2.96) to (-4.1,2.2)
(-2,2.96) to (-2.5,0.8);
\draw(-1.7,1.5)node{$Z_i$};
\draw[thick,cyan]
(-0.5,0) to (-3.5,3);
\fill[fill=gray!10]
(-4,0) rectangle (0,-0.3)
(-4,3.3) rectangle (0,3);
\draw[red]
(-1.8,2.4)node{$\teta_i$}
(-0.3,2)node{$\teta_{t(a_1)}$}
(0.2,0.8)node{$\teta_{t(a_l)}$}
(-0.8,1)node{$\ddots$}
(-2.5,2.4)node{$\ddots$}
(-2.9,1)node{$\teta_{t(a'_1)}$}
(-3.8,2.7)node{$\teta_{t(a'_m)}$};
\draw[cyan](-3.1,1.9)node{$\HR(\teta_i)$};
\draw
(-2,0)\rpt
(-2,3)\rpt
(-0.5,0)\bpt
(-3.5,3)\bpt;
\end{tikzpicture}
\caption{Half rotation of simple closed arcs in $\mfA^*$.}
\label{flip of simples}
\end{figure}
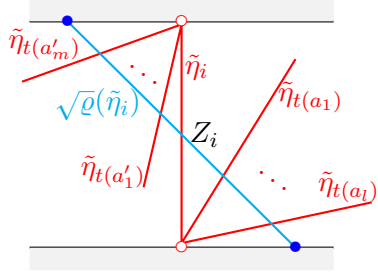

\end{proof}

Thus, $\oX(\HF(\tbeta_i))$ is a $\oK$-projective resolution of $S_i$, which is denoted by $\mfp S_i$.


We follow the definition of Koszul duality of dg algebras in \cite[Sec. 9]{Ke94}. 

\begin{definition}\label{def:koszul}
Let $\sS=S_1\oplus\cdots\oplus S_n$. The \emph{Koszul duality} $\Lambda_{\mfA}^*$ of $\Lambda_{\mfA}$ is the derived endomorphism algebra of $\sS$, that is, 
$$\Lambda_{\mfA}^{*}=\mcHom_{\Lambda_{\mfA}}(\mfp\sS,\mfp\sS).$$
The \emph{Koszul functor} is the triangle equivalence (see \Cref{thm:keller})
$$\mcHom_{\Lambda_{\mfA}}(\mfp\sS,-):\thick_{\mcD(\Lambda_{\mfA})}(\sS)\xrightarrow{\simeq}\per\Lambda_{\mfA}^{*}.$$
\end{definition}

By \Cref{lem:S}, $\oX(\HR(\mfA^*))$ is a $\oK$-projective resolution of $\sS$. So we have  $$\Lambda_{\mfA}^*=\mcHom_{\Lambda_{\mfA}}(\oX(\HR(\mfA^*)),\oX(\HR(\mfA^*))).$$
By \Cref{lem:str f}, $\HR(\mfA^*)$ is strongly formal with respect to $\mfA$. Hence by \eqref{eq:=1}, we have $\thick_{\mathcal{D}(\Lambda_{\mfA})}(\sS)=\per\Lambda_{\mfA}$. Thus, the Koszul functor can be written as
$$\mcHom_{\Lambda_{\mfA}}(\oX(\HR(\mfA^*)),-):\per\Lambda_{\mfA}\xrightarrow{\simeq}\per\Lambda_{\mfA}^{*}.$$

Recall from \Cref{def:quad dual} that $\Lambda_{\mfA}^{!}$ is the quadratic duality of $\Lambda_{\mfA}$.

\begin{lemma}\label{lem:quasi koszul}
    There is a quasi-isomorphism from $\Lambda_{\mfA}^{!}$ to $\Lambda_{\mfA}^*$ of dg algebras.
\end{lemma}

\begin{proof}
     Since $\HR(\mfA^*)$ is strongly formal with respect to $\mfA$, by \Cref{lem:quasi}, there is a quasi-isomorphism $$q_{\HR(\mfA^*)}:\Lambda_{\HR(\mfA^*)}=H^*(\Lambda_{\mfA}^*)\to \Lambda_{\mfA}^*.$$
    By definition, we have $\Lambda_{\HR(\mfA^*)}=\Lambda_{\mfA^*}$. Then by \Cref{prop:quad dual}, we get the required quasi-isomprhism.
\end{proof}

Due to the lemma above, the quadratic duality of $\Lambda_{\mfA}$ is also called the Koszul duality of $\Lambda_{\mfA}$. Hence, the dual $\mfA^*$ of a full formal open arc system $\mfA$ can be regarded as a geometric realization of the Koszul duality $\Lambda_{\mfA}^*$ of the graded gentle algebra $\Lambda_{\mfA}$ associated with $\mfA$. Similarly, we also call the composition 
\begin{equation}\label{eq:koszul}
    \mcK_{{\mfA}}=\rest{q_{\HR(\mfA^*)}}\circ\mcHom_{\Lambda_{\mfA}}(\mfp\sS,-):\per\Lambda_{\mfA}\to \per\Lambda_{\mfA^*}
\end{equation}
the Koszul functor. Let $\mcK^{{\mfA}}$ be a quasi-inverse of the Koszul functor $\mcK_{{\mfA}}$. By the dual of \Cref{cons:obj}, there is an injective map $\oY:\CC \rightarrow \per \Lambda_{\mfA^*}$.

\begin{theorem}\label{thm:koszul}
For any graded closed arc $\tbeta$, there is an isomorphism in $\per \Lambda_{\mfA}$
\begin{equation}\label{eq:main}
\oX(\HR(\tbeta)) \cong \mcK^{{\mfA}}(\oY(\tbeta)).
\end{equation}
That is, there is a commutative diagram
\[
\xymatrix@C=3pc{
\CC \ar[r]^{\oY}\ar[d]_{\HR} &   \per \Lambda_{\mfA^*}\ar[d]^{\mcK^{{\mfA}}} \\
\OC   \ar[r]_{\oX}& \per\Lambda_{\mfA}\ \ .
}
\]
\end{theorem}

\begin{proof}
    In \Cref{thm:independence}, if we take $\mfA'=\HR(\mfA^*)$, then $\oY =\oX'\circ \HR$. Comparing \eqref{eq:Phi} and \eqref{eq:koszul}, we have $\Phi_{\mfA'}=\mcK_{{\mfA}}$. Hence for $\tsigma=\HR(\tbeta)$ the isomorphism \eqref{eq:cong} becomes 
$\mcK_{{\mfA}}(\oX(\HR(\tbeta)))\cong \oY(\AHR(\tsigma))=\oY(\tbeta)$. Applying $\mcK^{{\mfA}}$, we get the required isomorphism.
\end{proof}

Therefore, the anti-clockwise half rotation is a geometric realization of the Koszul functor.





\begin{remark}
For a graded gentle algebra $\Lambda_{\mfA}$, a triangle equivalence between $\per\Lambda_{\mfA}$ and $\per\Lambda_{\mfA}^!$ is also established \cite[Prop. 4.4]{CJS} by using the method in \cite{HKK}.
\end{remark}

Let $\tsigma$ be a graded open arc and $\tiota$ be a graded closed arc. Since any intersection between $\tsigma$ and $\tiota$ occurs in the interior $\mfS^\circ$, by \eqref{eq:hkk1}, we have 
\begin{equation}\label{eq:symm}
\Int^{\rho}(\tsigma,\tiota)=\Int^{1-\rho}(\tiota,\tsigma),
\end{equation}
for any $\rho \in \mbZ$. This is not true for two graded open arcs in general. However, using the half rotation, we can unify these two cases by the following observation.

\begin{lemma} \label{lem:int-dim}
    Let $\tiota$ be a graded closed arc and $\tsigma$ a graded open arc. Then for any $\rho \in \mbZ$, we have
    \begin{equation}\label{eq:int transfer}
        \Int^{\rho}(\tsigma,\tiota)=\Int^{\rho}(\tsigma,\HR(\tiota)).
    \end{equation}
\end{lemma}

We give the following intersection-dimension formula involving the Koszul functor.

\begin{theorem}\label{thm:int=dim}
Let $\tsigma$ be a graded open arc and $\tiota$ be a graded closed arc. Then for any $\rho \in \mbZ$, we have
\begin{equation}\label{eq:intdim}
\Int^{\rho}(\tsigma,\tiota)=\operatorname{dim}\Hom_{\per\Lambda_{\mfA}}(\oX(\tsigma),\mcK^{{\mfA}}(\oY(\tbeta)){[\rho]}).
\end{equation}
\end{theorem}

\begin{proof}
By \eqref{eq:int=dim}, we have
$$\Int^{\rho}(\tsigma,\HR(\tiota))=\operatorname{dim}\Hom_{\per \Lambda_{\mfA}}(\oX(\tsigma),\oX(\HF(\tiota)){[\rho]}).$$
Hence the required formula follows from \eqref{eq:int transfer} and \eqref{eq:main}.
\end{proof}

Dually, we have a Koszul functor
$$\mcK_{\mfA^*}:\per\Lambda_{\mfA^*}\to\per\Lambda_{(\mfA^*)^*}$$
with an intersection-dimension formula
\begin{equation}\label{eq:intdim2}
\Int^{\rho}(\tiota,\tsigma)=\operatorname{dim}\Hom_{\per\Lambda_{\mfA^*}}(\oY(\tbeta),\mcK^{{\mfA^*}}(\oZ(\tsigma)){[\rho]}),
\end{equation}
where $\oZ$ is the injective map in \Cref{cons:obj} with respect to the open full formal arc system $(\mfA^*)^*$. Note that $(\mfA^*)^*=\mfA[1]$ and one can identify $\per\Lambda_{\mfA[1]}$ with $\per\Lambda_{\mfA}$ in a natural way that $\oX=\oZ$. Combining \eqref{eq:symm}, \eqref{eq:intdim} and \eqref{eq:intdim2}, we get the following consequence.

\begin{corollary}
    Let $\tsigma$ be a graded open arc and $\tiota$ be a graded closed arc. Then there is an isomorphism of vector spaces
    $$\Hom_{\per\Lambda_{\mfA}}(\oX(\tsigma),\mcK^{{\mfA}}(\oY(\tbeta)))\cong\Hom_{\per\Lambda_{\mfA^*}}(\oY(\tbeta),\mcK^{{\mfA^*}}(\oX(\tsigma)){[1]}).$$
\end{corollary}


\begin{thebibliography}{99}
\newcommand{\au}[1]{\textrm{#1},}
\newcommand{\ti}[1]{\textrm{#1},}
\newcommand{\jo}[1]{\textit{#1}}
\newcommand{\vo}[1]{\textbf{#1}}
\newcommand{\yr}[1]{(#1)}
\newcommand{\pp}[2]{#1--#2.}
\newcommand{\arxiv}[1]{\href{http://arxiv.org/abs/#1}{arXiv:#1}}
\bibitem[Am]{Am} \au{C.~Amiot} \ti{Indecomposable objects in the derived category of skew-gentle algebras using orbifolds} \arxiv{2107.02646}.
\bibitem[AB]{AB} \au{C.~Amiot \and T.~Br\"{u}stle} \ti{Derived equivalences between skew-gentle algebras using orbifolds} \jo{Doc. Math.} \vo{27} \yr{2022} \pp{933}{982} (\arxiv{1912.04367})
\bibitem[APS]{APS19}  \au{C.~Amiot, P-G.~Plamondon \and S.~Schroll}  \ti{A complete derived invariant for gentle algebras via winding numbers and Arf invariants}  \jo{Selecta Math. (N.S.)}  \vo{29}   \yr{2023}, no. 2, paper no. 30, 36 pp.  (\arxiv{1904.02555})
\bibitem[ALP]{ALP} \au{K.K.~Arnesen, R.~Laking \and D.~Pauksztello}  \ti{Morphisms between indecomposable complexes in the bounded derived category of a gentle algebra}  \jo{J. Algebra}  \vo{467}  \yr{2016}   \pp{1}{46}  (\arxiv{1411.7644})
\bibitem[BGS]{BGS96} \au{A.~Beilinson, V.~Ginzburg \and  W.~Soergel}  \ti{Koszul duality patterns in representation theory}  \jo{J. Amer. Math. Soc.}  \vo{9}  \yr{1996}, no. 2,  \pp{473}{527}
\bibitem[BQ]{BQ} \au{T.~Br\"{u}stle \and Y.~Qiu}  \ti{Tagged mapping class groups: Auslander-Reiten translation}   \jo{Math. Z.}   \vo{279}   \yr{2015}, no. 3-4,  \pp{1103}{1120}  (\arxiv{1212.0007})
\bibitem[CHS]{CHS}\au{W.~Chang, F.~Haiden \and S.~Schroll}  \ti{Braid group actions on branched coverings and full exceptional sequences}  \arxiv{2301.04398}.
\bibitem[CJS]{CJS}\au{W.~Chang, H.B.~Jin \and S.~Schroll}  \ti{Recollements of partially wrapped Fukaya categories and surface cuts}   \arxiv{2206.11196}.
\bibitem[CS1]{CS0} \au{W.~Chang \and S.~Schroll} \ti{A geometric realization of silting theory for gentle algebras} \jo{Math. Z.} \vo{303} \yr{2023} no. 3, Paper no. 67, 37pp. (\arxiv{2012.12663})
\bibitem[CS2]{CS}\au{W.~Chang \and S.~Schroll}  \ti{Exceptional sequences in the derived category of a gentle algebra}  \jo{Selecta Math. (N.S.)}  \vo{29}  \yr{2023} no. 3, paper no. 33, 44 pp.  (\arxiv{2205.15830})
\bibitem[HKK]{HKK} \au{F.~Haiden, L.~Katzarkov \and M.~Kontsevich} \ti{Flat surfaces and stability structures} \jo{Publ. Math. Inst. Hautes \'{E}tudes Sci.}  \vo{126}  \yr{2017}, no.1  \pp{247}{318}  (\arxiv{1409.8611})
\bibitem[IQ]{IQ2} \au{A.~Ikeda \and Y.~Qiu}  \ti{$q$-Stability conditions via $q$-quadratic differentials for Calabi-Yau-$\mathbb{X}$ categories}   \jo{Mem. Amer. Math. Soc.} to appear.
 (\arxiv{1812.00010})
\bibitem[IQZ]{IQZ}\au{A.~Ikeda, Y.~Qiu \and Y.~Zhou} \ti{Graded decorated marked surfaces: Calabi-Yau-$\mathbb{X}$ categories of gentle algebras}  \arxiv{2006.00009}.
\bibitem[JSW]{JSW} \au{H.B.~Jin, S.~Schroll \and Z.F.~Wang} \ti{A complete derived invariant and silting theory for graded gentle algebras} \arxiv{2303.17474}.
\bibitem[Ke1]{Ke94} \au{B.~Keller}  \ti{Deriving {D}{G} categories} \jo{Ann. Sci. \'Ecole Norm. Sup.} \vo{27}  \yr{1994}, no. 1,   \pp{63}{102}
\bibitem[Ke2]{Ke01} \au{B.~Keller} \ti{$A_{\infty}$-algebras in representation theory} \jo{Representations of Algebras} \vo{1}  \yr{2000} \pp{74}{86}
\bibitem[Ke3]{Ke06} \au{B.~Keller} \ti{On differential graded categories}  \jo{International Congress of Mathematicians}   \vo{II} Eur. Math. Soc. Z\"urich,  \yr{2006} \pp{151}{190}  (\arxiv{math/0601185})
\bibitem[LSV]{L-FSV}\au{D.~Labardini-Fragoso, S.~Schroll \and Y.~Valdivieso} \ti{Derived categories of skew-gentle algebras and orbifolds}  \jo{Glasg. Math. J.}  \vo{64}  \yr{2022} no. 3, \pp{649}{674}  (\arxiv{2006.05836})
\bibitem[LP]{LP} \au{Y.~Lekili \and A.~Polishchuk} \ti{Derived equivalences of gentle algebras via Fukaya categories} \jo{Math. Ann.}  \vo{376}  \yr{2020}, no. 1-2,   \pp{187}{255}  (\arxiv{1801.06370})
\bibitem[O]{O} \au{S.~Opper} \ti{On auto-equivalences and complete derived invariants of gentle algebras} \arxiv{1904.04859}.
  \bibitem[OPS]{OPS} \au{S.~Opper, P-G.~Plamondon \and S.~Schroll}  \ti{A geometric model for the derived category of gentle algebras} \arxiv{1801.09659}.
\bibitem[PPP]{PPP18} \au{Y.~Palu, V.~Pilaud \and  P-G. Plamondon}  \ti{Non-kissing and non-crossing complexes for locally gentle algebras}  \jo{J. Comb. Algebra}  \vo{3}  \yr{2019}, no. 4,  \pp{401}{438}  (\arxiv{1807.04730})
\bibitem[PP]{quadratic algebra}\au{A.~Polishchuk \and  L.~Positselski} \ti{Quadratic Algebras}  \jo{Univ. Lecture Ser.}   \vo{37}, American Mathematical Society, Providence, RI,    \yr{2005}, xii+159 pp.
\bibitem[Q1]{Q16}   \au{Y.~Qiu}   \ti{Decorated marked surfaces: Spherical twists versus braid twists} \jo{Math. Ann.} \vo{365}  \yr{2016} \pp{595}{633}  (\arxiv{1407.0806})
\bibitem[Q2]{Q18} \au{Y.~Qiu} \ti{Decorated marked surfaces (part B): topological realizations} \jo{Math. Z.}  \vo{288} \yr{2018} no. 1-2,  \pp{39}{53}
\bibitem[QZ]{QZ19} \au{Y.~Qiu \and Y.~Zhou} \ti{Decorated marked surfaces II: Intersection numbers and dimensions of Homs}  \jo{Trans. Amer. Math. Soc.}  \vo{372}  \yr{2019} \pp{635}{660}  (\arxiv{1411.4003})
\bibitem[QZZ]{QZZ}\au{Y.~Qiu, C.~Zhang \and Y.~Zhou} \ti{Two geometric models for graded skew-gentle Algebras} \arxiv{2212.10369v2}.
\end{thebibliography}
\end{document}